\documentclass[12pt]{article}
\usepackage[utf8]{inputenc}
\usepackage[a4paper]{geometry}
\usepackage{graphicx, wrapfig, subcaption, setspace, booktabs}
\usepackage[T1]{fontenc}
\usepackage[english]{babel}
\usepackage{amsmath}
\usepackage{amssymb}
\usepackage{lmodern}
\usepackage{amsthm}
\usepackage[shortlabels]{enumitem}
\addto\captionsenglish{}
\usepackage{mathtools}
\usepackage{microtype}
\usepackage{dirtytalk}

\DeclareMathOperator{\C}{\mathbb{C}}

\DeclareMathOperator{\Z}{\mathbb{Z}}

\DeclareMathOperator{\F}{\mathbb{F}}

\DeclareMathOperator{\omi}{\overline{\overline{\mu_{\textit{i}}}}}
\DeclareMathOperator{\omj}{\overline{\overline{\mu_{\textit{j}}}}}
\DeclareMathOperator{\on}{\overline{\overline{\nu}}}

\DeclareMathOperator{\FG}{\mathbb{F}G}
\DeclareMathOperator{\lm}{\lambda}

\DeclareMathOperator{\GL}{\text{GL}}
\DeclareMathOperator{\CGL}{\mathbb{C}\text{GL}}
\DeclareMathOperator{\FGL}{\mathbb{F}\text{GL}}

\DeclareMathOperator{\Hom}{Hom}
\DeclareMathOperator{\rank}{rank}
\DeclareMathOperator{\diag}{diag}
\DeclareMathOperator{\im}{Im}

\newtheorem{thm}{Theorem}[section]
\newtheorem{prop}[thm]{Proposition}

\newtheorem{lem}[thm]{Lemma}

\theoremstyle{remark}
\newtheorem{rem}[thm]{Remark}

\theoremstyle{definition}
\newtheorem{df}[thm]{Definition}
\newtheorem{exmp}[thm]{Example}

\title{ \textbf{\normalsize{COMPOSITION MULTIPLICITIES OF TENSOR PRODUCTS AND BRANCHING MULTIPLICITIES FOR $\GL(n)$}}}

\author{\normalsize \textsc{
		M.G. NORRIS}}
		\date{\small{University of Manchester} }
\begin{document}
\allowdisplaybreaks
\maketitle

\begin{abstract}
Following work of Brundan and Kleshchev \cite{BRUNDAN2000}, which considered tensor products with the natural module (and its dual) for $\GL(n)$, we take the next fundamental module and explore the relationship between multiplicities of composition factors of tensor products of simple modules for $\GL(n)$ and branching multiplicities. In doing this we provide a method for determining the composition multiplicities for some composition factors of tensor products of simple modules with the wedge square of the dual natural module.
\end{abstract} 

\section{\normalsize{INTRODUCTION}}

Understanding the tensor products of modules $M, M'$ for a group algebra along with the tensor product space $M^{\otimes r}$ is a fundamental problem in representation theory. One important part of this is understanding their composition series and in particular the multiplicities of composition factors. 

One case that is well understood is the composition factors of the tensor products of simple $\CGL_n(\C)$-modules. Here the multiplicities are given by the famous Littlewood-Richardson coefficients. The simple (rational) $\CGL_n(\C)$-modules are parametrised by (weakly) decreasing $n$-tuples of integers. Let $V_n(\lm)$ denote the simple $\CGL_n(\C)$-module corresponding to a decreasing $n$-tuple $\lm$. For suitable tuples $\lm, \mu$ and $\nu$ the Littlewood-Richardson rule gives us a way to input these tuples and output the multiplicity of the composition factor $V_n(\nu)$ in the tensor product $V_n(\lm) \otimes V_n(\mu)$ (see for example \cite{HOWE2012}). The Littlewood-Richardson coefficients also describe the composition multiplicities of the restriction of simple $\CGL_n(\C)$-modules to a Levi subgroup. We call the composition multiplicities arising in this way (from the restrictions to subgroups) branching multiplicities. 

If $\F$ is an algebraically closed field of positive characteristic then remarkably little is known in comparison about the composition factors of tensor products of (rational) $\FGL_n(\F)$-modules. Again the simple modules are parametrised by decreasing $n$-tuples of integers. Let $L_n(\lm)$ denote the simple $\FGL_n(\F)$-module corresponding to a decreasing $n$-tuple $\lm$. Let $V_n$ denote the natural $n$-dimensional $\FGL_n(\F)$-module and let $V_n^*$ denote its dual. In \cite{BRUNDAN2000} Brundan and Kleshchev study the composition multiplicities of $L_n(\lm) \otimes V_n^*$ for certain composition factors, more details of which are discussed below. The aim of this article is to say something about the composition multiplicities of $L_n(\lm) \otimes \bigwedge\nolimits^2 V_n^*$. In particular, we aim to describe a way to input a $\lm$ and a very particular $\mu$ and output the multiplicity of $L_n(\mu)$ as a composition factor of $L_n(\lm) \otimes \bigwedge\nolimits^2 V_n^*$. For $\lm = (\lm_1, \dots, \lm_n) \neq \underline{0}$ such that $\lm_n =0$ let $\lm_s$ be the last nonzero entry in $\lm$. For $1 \leq a \leq n$ let $\epsilon_a$ be the $n$-tuple with a $1$ in the $a^{th}$ entry and $0$ elsewhere. In Section \ref{mainsec} we describe a method which inputs the tuple $\lm$ and some $1 \leq i \leq s-1$ and outputs a value $t_s^i$ for which the following result holds. 

\begin{thm}\label{main}
Suppose the characteristic of $\F$ is a prime $p >2$. Let $\lm = (\lm_1, \dots, \lm_n) \neq \underline{0}$ be a decreasing $n$-tuple such that $\lm_s$ is the last nonzero entry in $\lm$ and $s <n-1$ . Suppose either $\lm_s =1$ or $s<n-2$. Finally, suppose that  $s-\lm_s \not\equiv n \mod p$ and for a fixed $1 \leq i \leq s-1$ that  $i-\lm_i \not\equiv n \mod p$. Then, $$[L_n(\lm) \otimes \bigwedge\nolimits^2 V^*_n : L_n(\lm - \epsilon_i - \epsilon_s)] = t_s^i.$$
\end{thm}

The idea behind the proof of this result comes from a \say{modular Littlewood-Richardson philosophy}. That is the idea that the multiplicities of composition factors of both tensor products of simple $\FGL_n(\F)$-modules and their restriction to standard Levi subgroups can be given in terms of similar combinatorics (see \cite{BRUNDAN1999}). This would mirror the case of simple $\CGL_n(\C)$-modules where the Littlewood-Richardson coefficients describe both. 

Consider the Levi subgroup for $\GL_n(\F)$ given by $\GL_{n-1}(\F) \times \GL_{1}(\F)$. In \cite{BRUNDAN2000} Brundan and Kleshchev define a $\GL_{n-1}(\F)$-invariant subspace $ L_n(\lm)^{(1)} \subset L_n(\lm)$. They show that under some conditions on the decreasing $n$-tuples $\lm$ and $\mu$, the multiplicity of $L_n(\mu)$ as a composition factor of $L_n(\lm) \otimes V_n^*$ is equal to the the multiplicity of $L_{n-1}(\overline{\mu})$ as a composition factor of $ L_n(\lm)^{(1)}$. Here $\overline{\mu}$ is the $n-1$-tuple derived from $\mu$ by removing the $n^{th}$ entry. The subspace $ L_n(\lm)^{(1)}$ is a sum of weight spaces whose dimensions have been described combinatorially in earlier work of Kleshchev in \cite{KLESH1997}. Combining this work in  \cite{BRUNDAN2000} the authors were able to describe the multiplicity of $L_n(\lm - \epsilon_i)$ as a composition factor of $L_n(\lm) \otimes V_n^*$ in terms of the combinatorics of $\lm$. 

In this paper we prove an analogous, but slightly more restrictive, result which relates the multiplicities of composition factors of $L_n(\lm) \otimes \bigwedge\nolimits^{2} V_n^*$ and composition factors of a $\GL_{n-2}(\F)$-invariant subspace of $L_n(\lm)$ which we denote by $L_n(\lm)^{(1,1) \times}$. For an $n$-tuple $\mu$ let $|\mu| = \sum_{i=1}^{n} \mu_i$ and let $\overline{\overline{\mu}}$ denote the $(n-2)$-tuple attained by dropping the last two entries of $\mu$. We write $\lm \nsim \mu$ if $\lm$ and $\mu$ are not linked in the sense of Section \ref{gln}. Our main result which is fundamental in the proof of Theorem \ref{main} is the following. 

\begin{thm}\label{key} 
Suppose the characteristic of $\F$ is a prime $p >2$. Let  $\mu, \lm \in \Z^n$ be decreasing $n$-tuples such that \begin{enumerate} \setlength\itemsep{-0.5em} \item $|\mu| = |\lm|-2$, \item $\mu_{n-2} = \mu_{n} =\lm_n=\lm_{n-1}=0$ and \item $\mu \nsim \lm - \epsilon_i - \epsilon_n$ for all $1 \leq i \leq n-1$ such that $ \lm - \epsilon_i - \epsilon_n$ is weakly decreasing. \end{enumerate}   Then $$[L_{n}(\lm) \otimes \bigwedge\nolimits^{2} V_n^*: L_{n}(\mu)] = [ L_{n}(\lm)^{(1,1) \times} : L_{n-2}(\overline{\overline{\mu}}) ].$$ \end{thm}

We expect that the method used in this paper could be applied to produce a generalisation of Theorem \ref{key} for tensor products $L_n(\lm) \otimes \bigwedge\nolimits^i V_n^*$ for arbitrary $i$, but the conditions on  $\lm$ and $\mu$ will become more and more restrictive. 

In Section 2 we fix notation and recall some preliminary results from the literature. Then in Section 3 we introduce and define the key objects of interest and prove an important isomorphism. In Section \ref{keysec} we prove Theorem \ref{key}. 

It is well known that the dimensions of all weight spaces for $L_n(\lm)$ are given in terms of ranks of matrices. In Section \ref{mainsec} we give an algorithm for constructing specific matrices. We then use this to prove Theorem \ref{main}. Although the use of the results in Section \ref{mainsec} to deduce Theorem \ref{main} parallel the use of Kleshchev's results in  \cite{KLESH1997} to describe the composition factors of $L_n(\lm) \otimes V_n^*$ in  \cite{BRUNDAN2000} they should not be compared with the results in \cite{KLESH1997} which are much stronger. 

\section*{\normalsize{\textit{Acknowledgements}}} Part of this work is based on work from the author's PhD. The author gratefully acknowledges Martin Liebeck for many helpful conversations during their PhD and for patiently reviewing preliminary versions of this work. The author would also like to acknowledge  Alexander Kleshchev for several significant conversations which guided this work. This work was supported both by the Engineering and Physical Sciences Research Council [EP/L015234/1] through the EPSRC Centre for Doctoral Training in Geometry and Number Theory (London School of Geometry and Number Theory) and by Heilbronn Institute for Mathematical Research via
the University of Manchester.

\section{\normalsize{NOTATION AND PRELIMINARIES}}

\subsection{The representation theory of $\GL_n(\F)$}\label{gln}

Let $p>0$ be the characteristic of the algebraically closed field $\F$. We will abbreviate $\GL_n(\F)$ to $\GL(n)$ and refer to $\FGL(n)$-modules as $\GL(n)$-modules. Let us briefly recall some of the structure of $\GL(n)$ detailed in sources such as \cite{JAN2003} and \cite{MALLE2011} and fix notation that will be used throughout. 

Let $y_{ij}$ for $1 \leq i,j \leq n$ denote the $n$ by $n$ matrix with a $1$ in the $(i,j)^{th}$ entry and $0$ elsewhere. The set $\{ y_{ij} | 1 \leq i,j \leq n\}$ forms a basis for the $\F$-algebra $M_n(\F).$ Let $x_{ij}: M_n(\F) \rightarrow \F$ for $1 \leq i,j \leq n$ be the maps that pick out the $(i,j)^{th}$ matrix coefficient for a matrix in $M_n(\F).$

Let $T(n)$ denote the torus of diagonal matrices in $\GL(n)$ and let $B(n)$ denote the Borel subgroup of upper triangular matrices in $\GL(n)$. Let $X(n)$ denote the character group of $T(n)$ which is generated by the characters $\epsilon_i := x_{ii}|_{T(n)}$ for $1 \leq i \leq n.$ Let $Y(n)$ denote the cocharacter group of $T(n)$ which is generated by the cocharacters $\epsilon'_i : a \mapsto a y_{ii} + \sum_{j \neq i} y_{jj}$ for all $a \in \F^{\times}.$ For $\lm \in X(n)$ and $\phi \in Y(n)$ there exists a unique integer $\langle \lm , \phi \rangle$ such that $\lm \circ \phi : a \mapsto a^{\langle \lm , \phi \rangle}$ for all $a \in \F^{\times}.$ This defines a bilinear pairing $\langle , \rangle : X(n) \times Y(n) \rightarrow \Z$ such that $\langle \epsilon_i, \epsilon'_j \rangle = \delta_{i,j}$.  We can identify both $X(n)$ and $Y(n)$ with the group $\Z^n$ under addition. In particular any character $\chi \in X(n)$ can be written as $\chi = a_1 \epsilon_1 + \dots + a_n \epsilon_n$ indicating that it corresponds to the map $\chi : \diag (t_1, \dots, t_n) \mapsto t_1^{a_1} \dots t_n^{a_n}.$ In this way $\chi$ is identified with the $n$-tuple $(a_1, \dots, a_n.)$ We say that a character $(a_1, \dots, a_n) \in X(n)$ is dominant if $a_i \geq a_{i+1}$ for $1 \leq i \leq n-1$ and we denote by $X^{+}(n)$ the set of dominant characters for $\GL(n).$

Let $W_n$ denote the Weyl group of $\GL(n)$ with respect to $T(n)$ whose action on $T(n)$ induces an action on $X(n)$ which permutes the entries of a character $(a_1, \dots, a_n) \in X(n).$ The root system $R(n)$ is of the form  $\{ \epsilon_i - \epsilon_j | 1 \leq i,j\leq n, i \neq j\}$ and the positive roots $R(n)^{+}$, determined by our choice of Borel, are the roots $\epsilon_i - \epsilon_j$ such that $i<j.$ Let $S(n) \subset R(n)^{+}$ denote the set of simple roots $\alpha_i := \epsilon_i - \epsilon_{i+1}$ for $1 \leq i <n.$ There is a dominance ordering on the set $X(n)$ in which we say for $\lm, \mu \in X(n)$ that $\mu$ is subdominant to $\lm$ and write $\mu \leq \lm$ if and only if $\lm - \mu \in \sum_{\alpha \in S(n)} \Z_{\geq 0} \alpha.$ Let $G_a := (\F, +)$ be the additive group of the field $\F$. For each $\alpha \in R(n)$ there exists a morphism of algebraic groups $u_{\alpha}: G_a \rightarrow \GL(n)$ which induces an isomorphism onto $u_{\alpha}(G_a)$ such that $t u_{\alpha}(c)t^{-1}= u_{\alpha}(\alpha(t)c)$ for all $t \in T(n)$ and $c \in \F$. We call $U_{\alpha} := \im(u_{\alpha})$ the root subgroup of $\GL(n)$ corresponding to $\alpha \in R(n)$. 

A $\GL(n)$-module $V$ has a direct sum decomposition $V = \bigoplus V_{\chi}$ where $V_{\chi} := \{v \in V |  \hspace{0.3cm} tv = \chi(t) v \hspace{0.3cm} \forall t \in T(n)\}$. We say $\chi \in X(n)$ is a weight for $V$ if $V_{\chi} \neq 0.$ We call  $\dim V_{\chi}$ the multiplicity of $\chi$ in $V$.

As defined in [II.2, \cite{JAN2003}]  the simple, Weyl and co-Weyl $\GL(n)$-modules are highest weight modules parametrised by  dominant weights $\lm \in X^{+}(n)$ which we denote $L_n(\lm), \Delta_n(\lm)$ and $\nabla_n(\lm)$ respectively. In each of $L_n(\lm), \Delta_n(\lm)$ and $\nabla_n(\lm)$  the weight $\lm \in X^{+}(n)$ occurs with multiplicity $1$ and all other weights for these modules are subdominant to $\lm$. Note that $L_n(\lm)$ is isomorphic to the socle of $\nabla_n(\lm)$ and the cosocle of $\Delta_n(\lm)$ and $\dim \nabla_n(\lm) = \dim \Delta_n(\lm)$. Also note that subtracting simple roots from a weight $\lm$ cannot decrease the last entry of $\lm$. From this we easily deduce the following lemma.

\begin{lem}\label{neg} Let $\lm$ be a weight in  $X^{+}(n)$ and suppose that $\lm_n \geq 0$. Then there is no weight $\mu$ for $L_n(\lm)$, $\nabla_n(\lm)$ or $\Delta_n(\lm)$ such that $\mu$ has a negative entry. \end{lem}

The natural $\GL(n)$-module, denoted by $V_n$, is the simple $\GL(n)$-module corresponding to the standard matrix action of $\GL(n)$ on the vector space $F^n$. Let $V_n^*$ denote its dual. Taking $e_1,\dots,e_n$ to be the standard basis of $F^n$ we note that each $e_i$ is an eigenvector for the action of $T$ on $V_n$ via the eigenvalue $\epsilon_i$. Since $e_1$ is clearly the highest weight vector we get $V_n \cong L_n(1,0,\dots,0)$. The set of all $e_{i_1} \wedge e_{i_2}$ with $1 \leq i_1 \leq i_2 \leq n$ defines a basis for $\bigwedge\nolimits^2V_n$ of eigenvectors of $T$ with eigenvalue $\epsilon_{i_1} + \epsilon_{i_2}$ respectively. Since the highest weight among these is $\epsilon_1 \wedge \epsilon_2$ we deduce $\bigwedge\nolimits^2V_n \cong L_n(1,1,0\dots,0)$. Similar arguments with the standard dual basis for $V_n^*$ show that $V_n^* \cong L_n(0,\dots,0,-1)$ and $\bigwedge\nolimits^2V_n^* \cong L_n(0,\dots,0,-1,-1).$

For a dominant weight $\lm \in X^+(n)$ and an element $\alpha \in \Z /p\Z$ we define the integer \begin{align*} 
cont_{\alpha}(\lm):=& | \{(a,b) | 1 \leq a \leq n, 0 < b \leq \lm_a, b-a  \equiv \alpha \mod p\}| \\ &-  | \{(a,b) | 1 \leq a \leq n, \lm_a < b \leq 0, b-a \equiv \alpha  \mod p\}|. \end{align*}  As in [\S 2, \cite{BRUNDAN2000}] we say that two dominant weights $\lm, \mu \in X^+(n)$ are linked and write $\lm \sim \mu$ if $cont_{\alpha}(\lm) = cont_{\alpha}(\mu)$ for all $\alpha \in \Z /p\Z.$ Note that if $\mu$ is in the affine Weyl orbit of $\lm$ then $\mu = \lm - ((\lm_i -i) -(\lm_j - j) - r p) (\epsilon_i - \epsilon_j)$ for some $r \in \Z$ and hence $\lm \sim \mu$. It therefore follows from the linkage principal (Corollary II.6.17 in \cite{JAN2003}) that if $\text{Ext}_{\GL(n)}^1( L_n(\lm), L_n(\mu)) \neq 0$ for some $\lm, \mu \in X^{+}(n)$ then $\lm \sim \mu$. Given $\mu \in X^{+}(n)$, we can decompose a $\GL(n)$-module $M$ as $M= M_1 \oplus M_2$ where $\lm \sim \mu$ for all weights $\lm$ of $M_1$  and $\nu \nsim \mu$ for all weights $\nu$ of $M_2$. 

\subsection{The hyperalgebra for $\GL_n(\F)$}\label{hypgln}

Let us finish this section by briefly recalling some details about the hyperalgebra, denoted $U(n,\F)$, for $\GL(n)$. In \cite{JAN2003} the hyperalgebra is called the algebra of distributions on $\GL(n)$ and denoted $Dist(G)$. The hyperalgebra for $\GL(n)$ can be constructed from $U(n, \C)$, the universal enveloping algebra of $Lie(\GL_n(\C))$. Recall that $y_{ij}$ is the matrix with a $1$ in the $(i,j)^{th}$ entry and $0$ elsewhere. The universal enveloping algebra $U(n, \C)$, is generated by the set $\{ y_{ij} | 1 \leq i, j \leq n\}$ subject to the relations \begin{equation}\label{urels} y_{ih} y_{kj} - y_{kj} y_{ih} = \delta_{kh} y_{ij} - \delta_{ij} y_{kh} \end{equation} for all $1 \leq i,h,k,j \leq n$ (see for example \cite{BRUND}).  

Denote by $U(n, \Z)$ the $\Z$-subalgebra of $U(n,\C)$ spanned by $$\{ y_{ij}^{(r)} , {{y_{ii}}\choose{r}} | 1 \leq i,j \leq n \hspace{0.5cm} i \neq j \hspace{0.5cm} r \geq 0 \}$$ where $ y_{ij}^{(r)} = \frac{y_{ij}^r}{r!}$ and ${{y_{ii}}\choose{r}} = \frac{y_{ii} (y_{ii}-1) \dots (y_{ii} - r+1)}{r!}$ are symbols in the divided power structure. 

The hyperalgebra of $\GL(n)$ can be defined as $U(n, \F) := U(n, \Z) \otimes \F$. For $i<j$ we will denote the element $y_{ij}^{(r)} \otimes 1_{\F}$ by $E_{ij}^{(r)}$ and the element $y_{ji}^{(r)} \otimes 1_{\F}$ by $F_{ij}^{(r)}$. Finally, let $H_{i}^{(r)}$ denote ${{y_{ii}}\choose{r}} \otimes 1_{\F}$ in $U(n, \F)$. When $r=1$ we will write $E_{ij} := E_{ij}^{(1)}$ and $F_{ij} :=F_{ij}^{(1)}$ , $H_{i} = H_{i}^{(1)}$. Define $U^{-}(n), U^{+}(n)$ and $U^0(n)$ to be the unital subalgebras of $U(n, \F)$ generated by $\{F_{ij}^{(r)}\}_{1 \leq i < j \leq n, \hspace{0.3cm} r \geq 0}$, $\{E_{ij}^{(r)}\}_{1 \leq i < j \leq n, \hspace{0.3cm} r \geq 0}$ and $\{H_{i}^{(r)}\}_{1 \leq i \leq n, \hspace{0.3cm} r \geq 0}$ respectively.  We can decompose $U(n,\F) =  U^{-}(n) U^0(n) U^{+}(n)$ (see for example [II.1.12(2),\cite{JAN2003}]). 

Let us recall some useful properties of elements of $U(n,\F)$ 

\begin{lem}\label{relhyp}
The following equalities hold in $U(n, \F)$: \\
For $1 \leq i< j <k \leq n$ and $r >0$, we have, 
\begin{align*} F_{ij}^{(r)} E_{ik} - E_{ik} F_{ij}^{(r)} &= F_{ij}^{(r-1)}E_{jk} \\
F_{ik}^{(r)} E_{ij} - E_{ij} F_{ik}^{(r)} &= F_{ik}^{(r-1)} F_{jk} \\
E_{ik}^{(r)} F_{jk} - F_{jk} E_{ik}^{(r)} &= E_{ik}^{(r-1)}  E_{ij} \\
E_{jk}^{(r)} F_{ik} - F_{ik} E_{jk}^{(r)} &=  E_{jk}^{(r-1)} F_{ij} \\
E_{ij}^{(r)} E_{jk} - E_{jk} E_{ij}^{(r)} &=  E_{ij}^{(r-1)} E_{ik} \\
E_{ij} F_{ij} - F_{ij}E_{ij} &= H_{i} - H_{j} \\
\end{align*}
and for $i \neq h$, $j \neq k$ and $i<j$
\begin{align*}
E_{ij} F_{hk} &= F_{hk} E_{ij} \hspace{2cm} \text{h<k}\\
E_{ij} E_{kh} &= E_{kh} E_{ij}  \hspace{2cm} \text{h>k}. 
\end{align*}
Finally, for $1 \leq i < j \leq n$ and $1 \leq a \leq n$ with $r>0$ we have, 
$$H_a F_{i,j}^{(r)} = F_{i,j}^{(r)} (H_a - \delta_{ai} r + \delta_{aj} r).$$
\end{lem}

\begin{proof} The final equality follows from the commutator formula which can be found in [II.1.20, \cite{JAN2003}]. The second, third and fourth last equalities follow directly from the relations given by Equation \ref{urels}. The first five follow from the divided power structure applied to the relations given by Equation \ref{urels}. For example it follows from Equation \ref{urels} that $F_{ij}^{r} E_{ik} = E_{ik}F_{ij}^r + r F_{ij}^{r-1}E_{jk}$ and therefore we have \begin{align*} F_{ij}^{(r)} E_{ik} & = \frac{1}{r!} F_{ij}^{r} E_{ik} \\ &=   \frac{1}{r!} E_{ik}F_{ij}^{r} +  \frac{1}{(r-1)!} F_{ij}^{r-1} E_{jk} \\ &= E_{ik}F_{ij}^{(r)} + F_{ij}^{(r-1)} E_{jk}. \end{align*} The other equalities follow similarly. \end{proof}

Any $\GL(n)$-module $M$ is naturally a $U(n,\F)$-module (see [I.7, \cite{JAN2003}]. Let us recall from [II.1.19, \cite{JAN2003}] that elements in $U(n,\F)$ act on the weight spaces $M_{\mu}$ for a $\GL(n)$-module $M$. For $\lm \in X(n)$ and $v \in M_{\lm}$ we have \begin{align*} H_i^{(r)} v &=  {{\langle \lm , \epsilon'_i \rangle}\choose{r}} v \\ E_{ij}^{(r)} v &\in M_{\lm + r (\epsilon_i - \epsilon_j)} \\ F_{ij}^{(r)} v &\in M_{\lm - r (\epsilon_i - \epsilon_j)}. \end{align*}

Let $X_1, \dots, X_n$ be the standard basis for $V_n^{*}$. In [II 2.16, \cite{JAN2003}] Jantzen describes the action of $E_{ab}$ and $F_{ab}$ on each $X_i$ as follows (see also Equation (1) in [$\S 5$, \cite{BRUNDAN2000}]) $$ E_{ab}^{(r)} X_i = -\delta_{r,1} \delta_{a,i} X_b \hspace{1cm} \text{and} \hspace{1cm} F_{ab}^{(r)} X_i = -\delta_{r,1} \delta_{b,i} X_a.$$ From this we can deduce the following lemma. 

\begin{lem}\label{Eab} The set $\{X_i \wedge X_j | 1 \leq i < j \leq n\}$ is a basis for $\bigwedge\nolimits^2 V_n^*$ such that $$ E_{ab}^{(r)} (X_i \wedge X_j) = - \delta_{r,1} \delta_{a,i} X_b \wedge X_j - \delta_{r,1} \delta_{a,j} X_i \wedge X_b$$ and  $$ F_{ab}^{(r)} (X_i \wedge X_j) = - \delta_{r,1} \delta_{b,i} X_a \wedge X_j -\delta_{r,1} \delta_{b,j} X_i \wedge X_a.$$ \end{lem}

\section{\normalsize{LEVELS AND THE FIXED POINT SPACE}}

Throughout this section we will assume $p$ is an odd prime and $\F$ is an algebraically closed field of characteristic $p$. In Proposition \ref{iso} in this section, we will construct, for certain $\lm \in X^{+}(n)$, a $\GL(n-2)$-isomorphism between the fixed point space of $L_n(\lm) \otimes V_n^*$ under the unipotent radical of a parabolic subgroup and a subspace $L_n(\lm)^{(1,1) \times} \subset L_n(\lm)$ which restricts to a $\GL(n-2)$ module. Before we proceed with defining the isomorphism we begin by detailing some key results about vectors of particular importance called primitive vectors. 

\subsection{Primitive vectors} 

Recall the notation for the hyperalgebra $U(n, \F)$, from Section \ref{hypgln} which is generated by the elements $E_{i,j}^{(r)}$, $F_{i,j}^{(r)}$ and $H_{i}^{(r)}$ for $1 \leq i < j \leq n$. 

Let $M$ be a $\GL(n)$-module and therefore a module for $U(n,\F)$. A vector $m \in M$ is said to be primitive if $E_{ij}^{(k)}m =0$ for $1 \leq i<j \leq n$ and $k > 0$. A highest weight module $M$ is generated by a weight vector with weight $\lm \in X^{+}(n)$ such that each weight of $M$ is subdominant to $\lm$. It follows from the action of $U(n,\F)$ of the weight space $M_{\lm}$ described in Section \ref{hypgln} that this is equivalent to $M$ being generated by a primitive weight vector $m$ of weight $\lm$. 

\begin{lem}[Lemma II.2.13 in \cite{JAN2003}] \label{prim}
Let $\lm$ be weight in $X^{+}(n)$. Then $\Delta_{n}(\lm)$ is generated by a primitive vector of weight $\lm$, and any $\GL(n)$-module $M$ generated by a primitive vector of weight $\lm$ is a homomorphic image of $\Delta_n(\lm)$.
\end{lem}

In particular the dimension of the space of all primitive vectors in a $\GL(n)$-module $M$ with weight $\lm$ is equal to $\dim \Hom_{\GL(n)}(\Delta_n(\lm),M).$ We can use this to describe the weights of primitive vectors but first let us make the following definitions. 

\begin{df}
A good filtration of a $\GL(n)$-module $M$ is an ascending filtration $0 = M_{(0)} < M_{(1)} < \dots < M_{(a)} = M$ for submodules $M_{(i)}$ of $M$, where each factor $M_{(i)}/ M_{(i-1)}$ is isomorphic to a direct sum of copies of $\nabla_n(\lm^{(i)})$ for some $\lm^{(i)} \in X^{+}(n)$.
\end{df}

We say that a weight $\epsilon$ is minuscule if there are no dominant weights $\lm \in X^{+}(n)$ such that $\lm < \epsilon$. With these definitions we can now recall the following result.
\begin{lem}[Lemma 4.8 in \cite{BRUNDAN2000}]\label{minweigt}
Let $\epsilon \in X^{+}(n)$ be a minuscule weight. For any $\lm \in X^{+}(n)$ the $\FG$-module $\nabla_n(\lm) \otimes L_n(\epsilon)$ has a good filtration with factors in $$\{ \nabla_n(\lm + w \epsilon) |\hspace{0.3cm} \text{for all} \hspace{0.3cm} w \in W_n \hspace{0.3cm} \text{such that} \hspace{0.3cm} \lm + w \epsilon \in X^{+}(n) \}$$ each occurring with multiplicity $1$. 
Furthermore for any $\mu \in X^{+}(n)$ we have \begin{equation*} \Hom_G(\Delta_n(\mu), \nabla_n(\lm) \otimes L_n(\epsilon)) = \begin{cases} \F \hspace{2cm} &\text{if} \hspace{0.5cm} \mu \in \lm + W_n \epsilon, \\ 0 &\text{otherwise.} \end{cases} \end{equation*}
\end{lem}

Recall that $V_n$ is the simple module with highest weight $(1,0,\dots,0)$ and $V_n^*$ is its dual which has highest weight $(0,\dots, 0, -1)$. The highest weight of $\bigwedge\nolimits^2V_n^*$ is the minuscule weight $(0,\dots,0,-1,-1)$. If $\lm \in X^{+}(n)$ then $1 \leq i \leq n$ is called $\lm$-removable if $\lm-\epsilon_i \in X^{+}(n)$. 

\begin{prop}\label{rem1}
For $\lm, \mu \in X^{+}(n)$ the space $ \Hom_{\GL(n)} (\Delta_n(\mu), \nabla_n(\lm) \otimes \bigwedge\nolimits^2 V_n^*)$ is zero unless $\mu = \lm - \epsilon_i -\epsilon_j$ for some $1 \leq i <j \leq n$ such that $j$ is $\lm$-removable and $i$ is $(\lm-\epsilon_j)$-removable. In this case $\dim \Hom_{\GL(n)} (\Delta_n(\mu), \nabla_n(\lm) \otimes \bigwedge\nolimits^2 V_n^*) = 1$.
\end{prop}

\begin{proof}
The orbit of  the weight $(0,\dots, 0, -1,-1)$ under $W_n$ is $\{ -\epsilon_i - \epsilon_j | 1 \leq i < j \leq n \}$. Therefore Lemma \ref{minweigt} implies that $ \Hom_{\GL(n)} (\Delta_n(\mu), \nabla_n(\lm) \otimes \bigwedge\nolimits^2 V_n^*)$ is zero unless $\mu = \lm - \epsilon_i -\epsilon_j$ for some $1 \leq i <j \leq n$. The result now follows from the fact that $\lm - \epsilon_i -\epsilon_j \in X^{+}(n)$ if and only if $j$ is $\lm$-removable and $i$ is $\lm-\epsilon_j$-removable.

\end{proof}

Let us make the following definition. 

\begin{df}\label{pair}
We say that a pair $(i,j)$ with $1 \leq i < j \leq n$ is $\lm$-removable if $j$ is $\lm$-removable and $i$ is $\lm-\epsilon_j$-removable.
\end{df}

If $v$ is a primitive vector of weight $\nu$ in $L_n(\lm) \otimes \bigwedge\nolimits^2V_n^*$ then by Lemma \ref{prim}  the submodule of $L_n(\lm) \otimes \bigwedge\nolimits^2V_n^*$ generated by $v$ is a homomorphic image of $\Delta_n(\nu)$ and hence $\dim \Hom_{\GL(n)}(\Delta_n(\nu), L_n(\lm) \otimes \bigwedge\nolimits^2V_n^*) \geq 1$. This implies that $\dim \Hom_{\GL(n)}(\Delta_n(\nu), \nabla_n(\lm) \otimes \bigwedge\nolimits^2V_n^*) \geq 1$. Applying Proposition \ref{rem1} we can deduce the following lemma.

\begin{lem}\label{primtensor}  If $v$ is a primitive vector in $L_n(\lm) \otimes \bigwedge\nolimits^2V_n^*$ of weight $\nu$ then $\nu = \lm - \epsilon_i - \epsilon_j$ for $(i,j)$ a $\lm$-removable pair. \end{lem}

\subsection{A key isomorphism} 

Let us assume that $\lm \in X^{+}(n)$, $\lm_{n} = \lm_{n-1} = 0$ and $|\lm| = \sum_{i=1}^{n} \lm_n = d$.  Furthermore we will assume $M$ is a submodule of $\nabla_n(\lm)$. Recall that this implies that the weights for $M$ are of the form $\lm - \sum_{i=1}^{n-1} a_i \alpha_i$ for $a_i \geq 0$ and by Lemma \ref{neg} they have no negative entries. 

Let $I:=\{\alpha_1, \dots, \alpha_{n-3}\} \subset S(n)$ and let $R_I = R(n) \cap \Z I$. Note that $R^{+}(n) \backslash R_I := \{\epsilon_i - \epsilon_n, \epsilon_j - \epsilon_{n-1} | 1 \leq i \leq n-1, 1 \leq j \leq n-2 \}.$ Then $U_{I}^{+} := \langle U_{\alpha} | \alpha \in  R^{+}(n) \backslash R_I \rangle$ is the unipotent radical of the standard parabolic subgroup $P_{I}^{+}$ with Levi subgroup $L_I := \GL(n-2) \times \GL(1) \times \GL(1)$. Let $T_2$ denote the torus $\GL(1) \times \GL(1)$. Since $\GL(n-2) \times T_2$ normalises $U_{I}^{+}$ the fixed point space $M^{U_{I}^{+}}$ has the structure of a $(\GL(n-2) \times T_2)$-module and hence also a $\GL(n-2)$-module (see [II.1.8, \cite{JAN2003}]). 

In this section we will construct a $\GL(n-2)$-homomorphism between a $\GL(n)$-module $M$ and its tensor product with $\bigwedge\nolimits^2V_n^*$. We will use this to prove the existence of an isomorphism between the fixed point space $(M \otimes \bigwedge\nolimits^2 V_n^*)^{U_{I}^{+}} $ and certain weight spaces of $M$. In the following sections we will use this isomorphism to deduce information about the composition factors of the tensor product when $M$ is simple.

We now define a map from the $\GL(n)$-module $M$ to its tensor product with $\bigwedge\nolimits^2 V_{n}^{*}$ using the elements of the hyperalgebra $U(n, \F)$ defined in Section \ref{hypgln}. 

\begin{df}
Let $X_1, \dots, X_n$ denote the standard basis for $V_{n}^{*}$ and define a map $e_{I}: M \rightarrow M \otimes \bigwedge\nolimits^2 V_{n}^{*}$ as follows:
\begin{align*}   m \mapsto \sum_{i < j <n-1} &\frac{1}{2} (E_{jn} E_{in-1} - E_{in}E_{jn-1})m \otimes X_i \wedge X_j - \sum_{i<n-1} E_{in} m \otimes X_i \wedge X_{n-1} \\ &+ \sum_{i< n-1} E_{in-1}m \otimes X_i \wedge X_n + m \otimes X_{n-1} \wedge X_n. \end{align*}
\end{df}

Before we prove the next proposition let us recall the following lemma which is a special case of Lemma I.7.14-16 in \cite{JAN2003}.

\begin{lem}[Lemma I.7.14-16 in \cite{JAN2003}] \label{prophyp} Let $M$ be a $\GL(n)$-module and let $U_{I}^{+}$ denote the unipotent radical defined above. 
\begin{enumerate} 
\setlength\itemsep{0em}
\item A subspace $M'$ of $M$ is a $\GL(n)$-submodule of $M$ if and only if it is a $U(n, \F)$-submodule.
\item Suppose $M'$ is a $\GL(n)$-module. Then a map $f: M \rightarrow M'$ is a $\GL(n)$-homomorphism if and only if it is a $U(n, \F)$-homomorphism. 
\item A vector $v \in M^{U_{I}^{+}}$ if and only if $E_{in}^{(r)} v =0$ and $E_{jn-1}^{(r)} v =0$ for $1 \leq i \leq n-1, 1 \leq j \leq n-2$ and $r>0$.
\end{enumerate}
\end{lem}

Before presenting the next proposition let us note that since the action of $\GL(n)$ on tensor products is diagonal it follows  from [I.7.11(1-2), \cite{JAN2003}] that for $\GL(n)$-modules $M, M'$ and $m \otimes m' \in M \otimes M'$ we have $$E_{ab}^{(k)}(m \otimes m') = \sum_{r=0}^k E_{ab}^{(k)} m \otimes E_{ab}^{(k-r)} m'$$ and $$F_{ab}^{(k)}(m \otimes m') = \sum_{r=0}^k F_{ab}^{(k)} m \otimes F_{ab}^{(k-r)} m'.$$ 

\begin{prop}\label{hom}
The map $e_{I}$ is an injective $\GL(n-2)$-homomorphism such that $(M \otimes \bigwedge\nolimits^2 V_n^*)^{U_{I}^{+}} \subset e_{I}(M).$ 
\end{prop}

 \begin{proof}
The set  $\{ X_i \wedge X_j | 1 \leq i < j \leq n\}$ forms a basis of $\bigwedge\nolimits^2 V_{n}^{*}$ and therefore $e_I (m)=0$ implies $m \otimes X_{n-1} \wedge X_n =0$ and hence $m=0$. Hence the map $e_I$ is injective. 

To see that $e_{I}$ is a $\GL(n-2)$-homomorphism we recall it is enough to show it is a $U(n-2, \F)$-homomorphism by Lemma \ref{prophyp}. Note that for any weight vector $m \in M$ the weights of $m$ and $e_{I}(m)$ relative to $T(n-2)$ are equal. So it suffices to show that $e_{I}(Xm) = Xe_{I}(m)$ for any $m \in M$ and $X \in U(n-2, \F)$ of the form $E_{ij}^{(k)}$ or  $F_{ij}^{(k)}$ with $1 \leq i <j <n-1$ and $k>0$. Using the divided power structure, the commutator relations detailed in Lemma \ref{relhyp} and the action of $E_{ij}^{(k)}$ and $F_{ij}^{(k)}$ on the basis $\{ X_i \wedge X_j | 1 \leq i < j \leq n\}$ detailed in Remark \ref{Eab} we deduce the following. For $1 \leq a < b \leq n-2$ and $k>0$ we have, 

\begin{align*}
E_{ab}^{(k)} e_{I}(m) &= \sum_{i<j<n-1} \frac{1}{2} (E_{ab}^{(k)}E_{jn}E_{in-1} - E_{ab}^{(k)}E_{in}E_{jn-1})m \otimes X_i \wedge X_j \\ 
& \hspace{0.4cm} + \sum_{i <n-1} \frac{1}{2} (E_{ab}^{(k-1)} E_{an}E_{in-1} - E_{ab}^{(k-1)} E_{in}E_{an-1})m \otimes X_b \wedge X_i \\ 
& \hspace{0.4cm} - \sum_{i<n-1} E_{ab}^{(k)}E_{in} m \otimes X_i \wedge X_{n-1} + E_{ab}^{(k-1)}E_{an}m \otimes X_b \wedge X_{n-1} \\   
& \hspace{0.4cm} + \sum_{i<n-1} E_{ab}^{(k)}E_{in-1} m \otimes X_i \wedge X_{n} -E_{ab}^{(k-1)}E_{an-1}m \otimes X_b \wedge X_{n} \\
& \hspace{0.4cm} + E_{ab}^{(k)}m \otimes X_{n-1} \wedge X_n \\ 
& = \sum_{i<j<n-1} \frac{1}{2} (E_{jn}E_{ab}^{(k)}E_{in-1} - E_{in}E_{ab}^{(k)}E_{jn-1})m \otimes X_i \wedge X_j \\ 
&\hspace{0.4cm} - \sum_{i <n-1} \frac{1}{2}  E_{ab}^{(k-1)} E_{in}E_{an-1} m \otimes X_b \wedge X_i  - \sum_{i<n-1} E_{in} E_{ab}^{(k)}m \otimes X_i \wedge X_{n-1}\\ 
& \hspace{0.4cm} + \sum_{i<n-1} E_{in-1} E_{ab}^{(k)}m \otimes X_i \wedge X_{n} + E_{ab}^{(k)}m \otimes X_{n-1} \wedge X_n \\
&=\sum_{i<j<n-1} \frac{1}{2} (E_{jn}E_{in-1} E_{ab}^{(k)}- E_{in}E_{jn-1}E_{ab}^{(k)})m \otimes X_i \wedge X_j \\
& \hspace{0.4cm} - \sum_{i<n-1} E_{in} E_{ab}^{(k)}m \otimes X_i \wedge X_{n-1}\\ 
& \hspace{0.4cm} + \sum_{i<n-1} E_{in-1} E_{ab}^{(k)}m \otimes X_i \wedge X_{n} + E_{ab}^{(k)}m \otimes X_{n-1} \wedge X_n \\
&=e_{I}(E_{ab}^{(k)} m)
\end{align*}
and
\begin{align*}
F_{ab}^{(k)} e_{I}(m) &= \sum_{i<j<n-1} \frac{1}{2} (F_{ab}^{(k)}E_{jn}E_{in-1} - F_{ab}^{(k)}E_{in}E_{jn-1})m \otimes X_i \wedge X_j \\ 
& \hspace{0.4cm} + \sum_{i <n-1} \frac{1}{2} (F_{ab}^{(k-1)} E_{bn}E_{in-1} - F_{ab}^{(k-1)} E_{in}E_{bn-1})m \otimes X_a \wedge X_i \\ 
& \hspace{0.4cm} - \sum_{i<n-1} F_{ab}^{(k)}E_{in} m \otimes X_i \wedge X_{n-1} + F_{ab}^{(k-1)}E_{bn}m \otimes X_a \wedge X_{n-1} \\   
& \hspace{0.4cm} + \sum_{i<n-1} F_{ab}^{(k)}E_{in-1} m \otimes X_i \wedge X_{n} -F_{ab}^{(k-1)}E_{bn-1}m \otimes X_a \wedge X_{n} \\
& \hspace{0.4cm}  + F_{ab}^{(k)}m \otimes X_{n-1} \wedge X_n \\ 
& = \sum_{i<j<n-1} \frac{1}{2} (E_{jn}F_{ab}^{(k)}E_{in-1} - E_{in}F_{ab}^{(k)}E_{jn-1})m \otimes X_i \wedge X_j \\ 
&\hspace{0.4cm} - \sum_{i <n-1} \frac{1}{2}  F_{ab}^{(k-1)} E_{in}E_{bn-1} m \otimes X_a \wedge X_i  - \sum_{i<n-1} E_{in} F_{ab}^{(k)}m \otimes X_i \wedge X_{n-1}\\ 
& \hspace{0.4cm} + \sum_{i<n-1} E_{in-1} F_{ab}^{(k)}m \otimes X_i \wedge X_{n} + F_{ab}^{(k)}m \otimes X_{n-1} \wedge X_n \\
&=\sum_{i<j<n-1} \frac{1}{2} (E_{jn}E_{in-1} F_{ab}^{(k)}- E_{in}E_{jn-1}F_{ab}^{(k)})m \otimes X_i \wedge X_j \\
& \hspace{0.4cm} - \sum_{i<n-1} E_{in} F_{ab}^{(k)}m \otimes X_i \wedge X_{n-1}\\ 
& \hspace{0.4cm} + \sum_{i<n-1} E_{in-1} F_{ab}^{(k)}m \otimes X_i \wedge X_{n} + F_{ab}^{(k)}m \otimes X_{n-1} \wedge X_n \\
&=e_{I}(F_{ab}^{(k)} m).
\end{align*}

It remains to show that $(M \otimes \bigwedge\nolimits^2 V_n^*)^{U_{I}^{+}}$ is contained in the image $e_{I}(M)$. For $x \in M \otimes \bigwedge\nolimits^2 V_n^*$ we can write $x = \sum_{i<j\leq n} m_{ij} \otimes X_i \wedge X_j$ for some $m_{ij} \in M$. Suppose $ x \in (M \otimes \bigwedge\nolimits^2 V_n^*)^{U_{I}^{+}}$.  By Lemma \ref{prophyp} this implies $E_{an} x= 0$ and $E_{bn-1} x =0$ for $1 \leq a \leq n-1$ and $1 \leq b \leq n-2$. Therefore we have, 
\begin{align*} 
0 &= E_{an} (\sum_{i<j} m_{ij} \otimes X_i \wedge X_j) \\ 
&= \sum_{i<j} E_{an} m_{ij} \otimes X_j \wedge X_j + \sum_{i<j} m_{ij} \otimes E_{an}(X_{i} \wedge X_j) \\
&= \sum_{i<j} E_{an} m_{ij} \otimes X_i \wedge X_j - \sum_{i<a} m_{ia} \otimes X_i \wedge X_n + \sum_{i>a} m_{ai} \otimes X_i \wedge X_n.
\end{align*}
Similarly, 
\begin{align*} 
0 &= E_{an-1} (\sum_{i<j} m_{ij} \otimes X_i \wedge X_j) \\ &= \sum_{i<j} E_{an-1} m_{ij} \otimes X_i \wedge X_j - \sum_{i<a} m_{ia} \otimes X_i \wedge X_{n-1} \\ 
&\hspace{0.4cm}+ \sum_{n-1>i>a} m_{ai} \otimes X_i \wedge X_{n-1} - m_{an} \otimes X_{n-1} \wedge X_n.
\end{align*}

Since the basis vectors  $X_i \wedge X_j$ are linearly independent this implies that $m_{an-1} = - E_{an}m_{n-1n},$  $ m_{an} = E_{an-1}m_{n-1n}$ and  $$E_{bn}(m_{an}) = \begin{cases} m_{ab} \hspace{1cm} &\text{if} \hspace{0.4cm} a<b \\ -m_{ba} \hspace{1cm} &\text{if} \hspace{0.4cm} b<a \end{cases}. $$ Therefore for any  $1 \leq a<b <n-1$ we have $m_{ab} = E_{bn}E_{an-1} m_{n-1n} = - E_{an}E_{bn-1} m_{n-1n}$. Hence $m_{ab} = \frac{1}{2} (E_{bn}E_{an-1} - E_{an}E_{bn-1}) m_{n-1n}$ and therefore $x = e_{I}(m_{n-1n}) \in e_{I}(M)$ and the proof is complete. 
\end{proof}

\begin{rem}\label{En-1n}
For $k>0$ we also have the following identity:
\begin{align*}
E_{n-1n}^{(k)} e_{I}(m) &= \sum_{i<j<n-1} \frac{1}{2} (E_{n-1n}^{(k)}E_{jn}E_{in-1} - E_{n-1 n}^{(k)}E_{in}E_{jn-1})m \otimes X_i \wedge X_j \\
& \hspace{0.4cm} - \sum_{i<n-1} E_{n-1n}^{(k)}E_{in} m \otimes X_i \wedge X_{n-1} + \sum_{i<n-1} E_{n-1n}^{(k-1)}E_{in} m \otimes X_i \wedge X_{n} \\
& \hspace{0.4cm} + \sum_{i<n-1} E_{n-1n}^{(k)}E_{in-1} m \otimes X_i \wedge X_{n} + E_{n-1n}^{(k)}m \otimes X_{n-1} \wedge X_n \\ 
&=\sum_{i<j<n-1} \frac{1}{2} (E_{jn}E_{n-1n}^{(k)}E_{in-1} - E_{in}E_{n-1 n}^{(k)}E_{jn-1})m \otimes X_i \wedge X_j \\
& \hspace{0.4cm} - \sum_{i<n-1} E_{in} E_{n-1n}^{(k)}m \otimes X_i \wedge X_{n-1} + \sum_{i<n-1} E_{in} E_{n-1n}^{(k-1)}m \otimes X_i \wedge X_{n} \\
& \hspace{0.4cm} + \sum_{i<n-1} E_{in-1} E_{n-1n}^{(k)}m \otimes X_i \wedge X_{n}  \\
& \hspace{0.4cm}- \sum_{i<n-1} E_{in}E_{n-1n}^{(k-1)} m \otimes X_i \wedge X_n + E_{n-1n}^{(k)}m \otimes X_{n-1} \wedge X_n \\ 
&= \sum_{i<j<n-1} \frac{1}{2} (E_{jn}E_{in-1} E_{n-1n}^{(k)}- E_{in}E_{jn-1}E_{n-1 n}^{(k)})m \otimes X_i \wedge X_j \\
& \hspace{0.4cm} - \sum_{i<j<n-1} \frac{1}{2} (E_{jn}E_{in}E_{n-1n}^{(k-1)} - E_{in}E_{jn}E_{n-1 n}^{(k-1)})m \otimes X_i \wedge X_j \\
& \hspace{0.4cm} - \sum_{i<n-1} E_{in} E_{n-1n}^{(k)}m \otimes X_i \wedge X_{n-1} + \sum_{i<n-1} E_{in-1} E_{n-1n}^{(k)}m \otimes X_i \wedge X_{n}  \\
& \hspace{0.4cm} + E_{n-1n}^{(k)}m \otimes X_{n-1} \wedge X_n \\ 
&=e_{I}(E_{n-1n}^{(k)} m).
\end{align*}
\end{rem}

The fixed point space in Proposition \ref{hom} will be one side of our isomorphism. Let us know describe the weight space of $M$ that will be the image. 

\begin{df}
The $(i,j)$-level of $M$, denoted $M^{(i,j)}$, is the sum of weight spaces $$M^{(i,j)}: = \bigoplus_{\substack{\mu_{n-1} =\lm_{n-1} + i \\ \mu_{n} = \lm_{n} + j}} M_{\mu}.$$ 
\end{df}

The $\GL(n)$-module $M$ admits an action of $\GL(n-2)$. It is clear that under the restriction to the Levi subgroup $\GL(n-2) \times T_2$, the subspace $M^{(i,j)}$ has the structure of a $\GL(n-2)$-submodule since there are no elements of $\GL(n-2)$ that act in a way that changes the last two parts of the weight corresponding to a vector in $M^{(i,j)}$. Therefore the $(i,j)$-level of $M$ can be naturally thought of as a (possibly zero) $\GL(n-2)$-module.  

For $i,j \geq 0$ define the subset $M^{(i,j) \times}$ of the $(i,j)$-level of $M$ as follows.
\begin{align*}
M^{(i,j) \times} :=\{ v \in M^{(i,j)}  &| E_{n-1n} v =0\}
\end{align*}
\begin{rem}
Let $\alpha_{n-1} = \epsilon_{n-1} - \epsilon_n$ and denote by $U_{\alpha_{n-1}}$ the corresponding root subgroup. Note that $M^{(i,j) \times} = M^{(i,j)} \cap M^{U_{\alpha_{n-1}}}$. We have already shown that $M^{(i,j)}$ is a $\GL(n-2)$-submodule of $M$ and it follows from the relations in Lemma \ref{relhyp} that $M^{U_{\alpha_{n-1}}}$ is a $U(n-2, \F)$-submodule of $M$. Therefore by Lemma \ref{prophyp}, $M^{(i,j) \times}$ is the intersection of $\GL(n-2)$-submodules and therefore a $\GL(n-2)$-module. Furthermore if $j=0$ then $E_{n-1n} M^{(i,j)} = 0$ since otherwise it would be contained in a weight space for $M$ corresponding to a weight with a negative entry. Therefore $M^{(i,0) \times} = M^{(i,0)}$. 
\end{rem}

Let us summarise some useful results regarding the subspace $M^{(i,j) }$ in the following lemma. 

\begin{lem}\label{zero}
Let $\lm \in X^{+}(n)$ such that  $\lm_{n-1} = \lm_n =0$ and let $M$ be a submodule of $\nabla_n(\lm)$. Let $W := M^{(0,0)} \oplus M^{(1,0)} \oplus M^{(0,1)} \oplus M^{(1,1)}$. Then, for $1 \leq a,b \leq n-2$ and $r+k >1$ we have \begin{enumerate} \setlength\itemsep{0em}
\item $E_{an}^{(k)}E_{bn}^{(r)} W =0$ and 
\item $E_{an-1}^{(k)}E_{bn-1}^{(r)}W=0$.
\end{enumerate}
\end{lem}

\begin{proof} 
Recall that if $M$ is a submodule of $\nabla_n(\lm)$ the weights $\mu \in X(n)$ for which $M_{\mu} \neq 0$ are of the form $\mu = \lm - \sum_{i=1}^{n-1} a_i \alpha_i$ where each $a_i \geq 0$. Note that $\mu = \lm - \sum_{i=1}^{n-1} a_i \alpha_i$ is in $M^{(i,j)}$ if and only if $\lm_{n-1} - a_{n-1} + a_{n-2} = \lm_{n-1} + i$ and  $\lm_n + a_{n-1} =\lm_n + j$. 

Let $v \in W$ be a nonzero weight vector of weight $\mu = \lm - \sum_{i=1}^{n-1} a_i \alpha_i$. Note that $0 \leq a_{n-1} \leq 1$ and $0 \leq a_{n-2} \leq 2$. Furthermore if $a_{n-1} = 0$ then $0 \leq a_{n-2} < 2$ implying that if $a_{n-2} = 2$ then $a_{n-1} = 1$. 

To see $1$ note that $E_{an}^{(k)} E_{bn}^{(r)} v \subset M_{\nu}$ where $\nu = \lm - \sum_{i=1}^{n-2} x_i \alpha_i + (r+k - a_{n-1}) \alpha_{n-1}$ for some $x_i \in \Z$. However, $(r+k - a_{n-1}) > 0 $ and hence $\nu$ is not subdominant to $\lm$, so $M_{\nu} = 0$ and $E_{an}^{(k})E_{bn}^{(r)} v =0$. 

To see $2$ now note that $E_{an-1}^{(k)}E_{bn-1}^{(r)} v \subset M_{\nu}$  where $\nu = \lm - \sum_{i=1}^{n-3} x_i \alpha_i + (r+k - a_{n-2}) \alpha_{n-2} - a_{n-1} \alpha_{n-1}$ for some $x_i \in \Z$. If $a_{n-2} < 2$ or $r+k>2$ then $(r+k - a_{n-2}) > 0$ and therefore $E_{an-1}^{(k)}E_{bn-1}^{(r)}  v =0$. Now assume $a_{n-2} = 2$ and $r+k = 2$. As noted above this implies $a_{n-1} = 1$ and $v \in M^{(1,1)}$. By our assumption that  $\lm_{n-1} = \lm_n =0$ we can write $\lm = \sum_{i=1}^{n-2} \lm_i \epsilon_i$. Therefore $\nu =  \sum_{i=1}^{n-2} \lm_i \epsilon_i -  \sum_{i=1}^{n-2} y_i \epsilon_i  - \epsilon_{n-1} + \epsilon_{n}$. Therefore by Lemma \ref{neg} $E_{an-1}^{(k)}E_{bn-1}^{(r)} v =0$. 
\end{proof}

We can now state and the following result. 

\begin{prop}\label{iso}
Let $\lm \in X^{+}(n)$ be such that $\lm_{n-1} = \lm_n =0$ and let $M$ be a submodule of $\nabla_n(\lm)$. Define $W^{\times} := M^{(0,0)} \oplus M^{(1,0)} \oplus M^{(0,1) \times} \oplus M^{(1,1) \times} $. Then, $W^{ \times} \cong (M \otimes \bigwedge\nolimits^2 V_n^*)^{U_{I}^{+}}$ as $\GL(n-2)$-modules.
\end{prop}

\begin{proof}

Let $\overline{e}_{I}$ denote the restriction of $e_{I}$  to $W^{ \times} $.  We begin by showing the image of $\overline{e}_{I}$ is contained in $(M \otimes \bigwedge\nolimits^2 V_n^*)^{U_{I}^{+}}$. 

Let $m \in W^{ \times} $. For $1 \leq a < n-1$ and $k>0$ we have 
\begin{align*}
E_{an}^{(k)}  e_{I}(m) &= \sum_{1 \leq i<j<n-1} \frac{1}{2} (E_{an}^{(k)}E_{jn}E_{in-1} - E_{a n}^{(k)}E_{in}E_{jn-1})m \otimes X_i \wedge X_j \\
& \hspace{0.4cm} - \sum_{i<n-1} \frac{1}{2} (E_{an}^{(k-1)}E_{an}E_{in-1} - E_{a n}^{(k-1)}E_{in}E_{an-1})m \otimes X_i \wedge X_n \\
& \hspace{0.4cm} -  \sum_{i<n-1} E_{an}^{(k)} E_{in} m \otimes X_{i} \wedge X_{n-1} - E_{an}^{(k-1)} E_{an} m \otimes X_{n-1} \wedge X_{n} \\
& \hspace{0.4cm} + \sum_{i<n-1} E_{an}^{(k)} E_{i n-1} m \otimes X_i \wedge X_n + E_{an}^{(k)} m \otimes X_{n-1} \wedge X_{n}. \\
\end{align*}
If $k>1$ this is zero since by  Lemma \ref{zero} all the terms are zero.  If $k=1$ then the terms left once we remove the terms which we know are zero by  Lemma \ref{zero} $(1)$ are the following:
\begin{align*}
E_{an} e_{I}(m) &= - \sum_{i<n-1} \frac{1}{2} (E_{an}E_{in-1} - E_{in}E_{an-1})m \otimes X_i \wedge X_n \\
& \hspace{0.4cm} - E_{an} m \otimes X_{n-1} \wedge X_n + \sum_{i<n-1} E_{an} E_{in-1} m \otimes X_i \wedge X_n + E_{an} m \otimes X_{n-1} \wedge X_n \\
&= \sum_{i<n-1} \frac{1}{2} ( E_{an}E_{in-1} + E_{in} E_{an-1}) m \otimes X_i \wedge X_n  \\
&=  \sum_{i<n-1} \frac{1}{2} ( (E_{an-1} E_{n-1n} - E_{n-1n}E_{an-1})E_{in-1} + E_{in} E_{an-1}) m \otimes X_i \wedge X_n \hspace{0.5cm}  \\
&=  \sum_{i<n-1} \frac{1}{2} ( E_{an-1} E_{n-1n} E_{in-1} + E_{in} E_{an-1}) m \otimes X_i \wedge X_n \hspace{0.5cm}   \\
&=  \sum_{i<n-1} \frac{1}{2} ( E_{an-1} (E_{in-1} E_{n-1 n} - E_{in}) + E_{in} E_{an-1}) m \otimes X_i \wedge X_n \hspace{0.5cm}    \\
&=  \sum_{i<n-1} \frac{1}{2} ( E_{an-1} E_{in-1} E_{n-1 n} ) m \otimes X_i \wedge X_n \hspace{0.5cm}  \\
\end{align*}
where we have again used the relations from Lemma \ref{relhyp} and Lemma \ref{zero}. Observe that it follows from the definition of $W^{\times}$ that $E_{n-1 n } m =0$. Therefore we have shown $E_{an}^{(k)} e_{I}(m) =0$ for $1 \leq a < n-1$,  $k>0$. If $a=n-1$ then, by Remark \ref{En-1n}, we have $E_{n-1n}^{(k)}  e_{I}(m) = e_{I}(E_{n-1n}^{(k)} m)$  and therefore $E_{n-1n}^{(k)}  e_{I}(m) =0$. 

For $1 \leq a < n-1$ and $k>0$ we have, 
\begin{align*}
E_{an-1}^{(k)}  e_{I}(m) &= \sum_{i<j<n-1} \frac{1}{2} (E_{an-1}^{(k)}E_{jn}E_{in-1} - E_{a n-1}^{(k)}E_{in}E_{jn-1})m \otimes X_i \wedge X_j \\
& \hspace{0.4cm} - \sum_{i<n-1} \frac{1}{2} (E_{an-1}^{(k-1)}E_{an}E_{in-1} - E_{a n-1}^{(k-1)}E_{in}E_{an-1})m \otimes X_i \wedge X_{n-1} \\
& \hspace{0.4cm} -  \sum_{i<n-1} E_{an-1}^{(k)} E_{in} m \otimes X_{i} \wedge X_{n-1} + \sum_{i<n-1} E_{an-1}^{(k)} E_{i n-1} m \otimes X_i \wedge X_n  \\
& \hspace{0.4cm} - E_{an-1}^{(k-1)} E_{a n-1} m \otimes X_{n-1} \wedge X_n + E_{an-1}^{(k)} m \otimes X_{n-1} \wedge X_{n}. \\
\end{align*}
If $k>1$ this is zero by Lemma \ref{zero} since all the terms in the sum are zero and if $k=1$ we have,
\begin{align*}
E_{an-1} e_{I_{n-1}}(m) &= - \sum_{i<n-1} \frac{1}{2} (E_{an}E_{in-1} - E_{in}E_{an-1})m \otimes X_i \wedge X_{n-1} \\
& \hspace{0.4cm} - \sum_{i<n-1} E_{an-1} E_{in} m \otimes X_{i} \wedge X_{n-1} - E_{a n-1} m \otimes X_{n-1} \wedge X_n \\
& \hspace{0.4cm} + E_{an-1} m \otimes X_{n-1} \wedge X_{n} \\
&= - \sum_{i<n-1} \frac{1}{2} (E_{an}E_{in-1} + E_{in}E_{an-1}) m \otimes X_i \wedge X_{n-1}  \\
&= - \sum_{i<n-1} \frac{1}{2} ( (E_{an-1} E_{n-1n} - E_{n-1n}E_{an-1})E_{in-1} + E_{in} E_{an-1}) m \otimes X_i \wedge X_{n-1} \hspace{0.5cm}   \\
&= - \sum_{i<n-1} \frac{1}{2} ( E_{an-1} E_{n-1n} E_{in-1} + E_{in} E_{an-1}) m \otimes X_i \wedge X_{n-1} \hspace{0.5cm} \\
&= - \sum_{i<n-1} \frac{1}{2} ( E_{an-1} (E_{in-1} E_{n-1 n} - E_{in}) + E_{in} E_{an-1}) m \otimes X_i \wedge X_{n-1} \hspace{0,5cm}    \\
&= -  \sum_{i<n-1} \frac{1}{2} ( E_{an-1} E_{in-1} E_{n-1 n} ) m \otimes X_i \wedge X_{n-1} \hspace{0.5cm}   \\
\end{align*}
where we have again used the relations from Lemma \ref{relhyp} and Lemma \ref{zero}. Since $E_{n-1n}m=0$ we have shown $E_{an-1}^{(k)} e_{I}(m) =0$ for $1 \leq a < n-1$, $k>0$. Since $E_{an}^{(k)} e_{I}(m) =0$ and $E_{bn-1}^{(k)} e_{I}(m) =0$ for $1 \leq a \leq n-1$, $1 \leq b \leq n-2$ and $k>0$ by Lemma \ref{prophyp} we have  $e_I(m) \in (M \otimes \bigwedge\nolimits^2 V_n^*)^{U_{I}^{+}}$. 

It remains to show that the restriction map $\overline{e_I}$ is surjective. Suppose $v$ is a vector in $(M \otimes \bigwedge\nolimits^2 V_n^*)^{U_{I}^{+}}$. Then it can be written as a sum of weight vectors $v = \sum_i x_i$. 
Since each $x_i$ is itself in $(M \otimes \bigwedge\nolimits^2 V_n^*)^{U_{I}^{+}}$ it is enough to show that for each weight vector $x \in (M \otimes \bigwedge\nolimits^2 V_n^*)^{U_{I}^{+}}$ there is a weight vector $m \in W^{\times}$ such that $x = e_I(m)$.

Let $x \in (M \otimes \bigwedge\nolimits^2 V_n^*)^{U_{I}^{+}}$  be a weight vector. It follows from Proposition \ref{hom} and the definition of the map $e_i$ that there exists $m \in M$ such that $x = e_I(m)$ and that $m$ is a weight vector. Since $x \in (M \otimes \bigwedge\nolimits^2 V_n^*)^{U_{I}^{+}}$  it follows from Lemma \ref{prophyp} and Remark \ref{En-1n} that $e_I( E_{n-1n}^{(k)} m)  = E_{n-1n}^{(k)} e_I (m) = 0$ for $k>0$.  Since $e_I$ is injective by Proposition \ref{hom} we have shown $E_{n-1n}^{(k)} m =0$ for all $k>0$.  

Suppose $m$ is $\GL(n-2)$-primitive. Then $E_{ij}^{(k)} m =0$ for all $1 \leq i < j \leq n-2$ and $k>0$. By Proposition \ref{hom} we have $E_{ij}^{(k)} x =e_I( E_{ij}^{(k)} m) =0$ for all $1 \leq i < j \leq n-2$ and $k>0$. Furthermore since $x \in (M \otimes \bigwedge\nolimits^2 V_n^*)^{U_{I}^{+}}$ by  Lemma \ref{prophyp} we have $E_{in-1}^{(k)} x =0$ and  $E_{jn}^{(k)} x =0$ for $1 \leq i \leq n-2$, $1 \leq j \leq n-1$ and $k>0$. Therefore $x$ is $\GL(n)$-primitive. Combining Proposition \ref{rem1} and Lemma \ref{prim} we can deduce that $x$ has weight $\lm - \epsilon_i -\epsilon_j$ for some $1 \leq i<j \leq n$. By the definition of the map $e_I$ this implies $m$ is a weight vector of weight $\lm - \epsilon_i -\epsilon_j + \epsilon_{n-1} + \epsilon_n$. Hence $m \in  W^{ \times}.$

Now suppose $m$ is not $\GL(n-2)$-primitive. Then there exists some product of elements of the form $E_{ab}^{(k)}$ with $1 \leq a<b \leq n-2$ and $k>0$ which applied to $m$ gives a $\GL(n-2)$-primitive vector. We will denote this particular product by $E_{AB}^{(K)}$ and let $m_{prim} := E_{AB}^{(K)}m$. We will now show that $m$ and $m_{prim}$ are in the same level $M^{(i,j)}.$

Let $m' := E_{ab}^{(k)} m$ for some $1 \leq a<b \leq n-2$ and $k>0$. Recall from for Section \ref{hypgln} that if $m \in M_{\nu}$ for some weight $\nu$ for $M$ then $E_{ab}^{(k)} M_{\nu} \subset M_{\nu + k( \alpha_{a} - \alpha_{b})}$. Since $a, b <n-1$, $\nu$ and $\nu + k( \alpha_a - \alpha_b)$ have the same last two entries and hence $m'$ and $m$ are in the same level $M^{(i,j)}$. Since we have shown $m_{prim}  \in W^{\times}$ and $E_{n-1n}^{(k)} m =0$ for all $k>0$ it follows that $m \in W^{\times}$. This completes the inductive argument and therefore we have shown $\overline{e}_I$ is a surjective, injective $\GL(n-2)$-homomorphism and therefore an isomorphism.  
\end{proof}

\section{\normalsize{COMPOSITION MULTIPLICITIES OF TENSOR PRODUCTS AND BRANCHING MULTIPLICITIES}}\label{keysec}

We continue to assume $p$ is an odd prime. We can now prove the key result relating composition factors of tensor products to branching multiplicities for the restriction of irreducible $\FGL(n)$-modules to $\GL(n-2)$. Recall that $I:=\{\alpha_1, \dots, \alpha_{n-3}\}$ and $U^{+}_I$ is the unipotent radical of the standard parabolic subgroup of $\GL(n)$ corresponding to $I$. Furthermore $L_I \cong \GL(n-2) \times T_2$, is the corresponding Levi subgroup, normalised by $U^{+}_I$. Let us recall the following result which is proved for general $I \subset S(n)$  in [II.2.11, \cite{JAN2003} ] but we only need applied to our choice of $I$. 

\begin{prop}\label{janprop}
Let $I \subset S(n)$ and $\lm \in X^{+}(n)$. Then:
\begin{enumerate}
\setlength\itemsep{0em}
\item $L_n(\lm)^{U^{+}_I} = \bigoplus_{\nu \in \Z I} L_n(\lm) _{\lm- \nu}$.
\item $\bigoplus_{\nu \in \Z I} L_n(\lm) _{\lm- \nu}$ is a simple $L_I$-module with highest weight $\lm$.
\end{enumerate}
\end{prop}

We should note that in the case that $I:=\{\alpha_1, \dots, \alpha_{n-3}\}$ considered below any $\GL(n-2)$-submodule of $L_n(\lm)^{U^{+}_I}$ is also a $\GL(n-2) \times T_2$-submodule. Therefore by Proposition \ref{janprop} $L_n(\lm)^{U^{+}_I}$ is a simple $\GL(n-2)$-module. 

We say that a $\FGL(n)$-module $M$ is polynomial of degree $d$ if every nonzero weight space of $M$ is of weight $\mu$ such that $\mu_i \geq 0$ for $1 \leq i \leq n$ and $|\mu| = d$. 
For a weight $\mu \in X^{+}(n)$ we denote by $\overline{\overline{\mu}}$ the element of $X^{+}(n-2)$ obtained by deleting the last two entries of $\mu$. 
\begin{lem}\label{poly1}
Suppose $\mu \in X^{+}(n)$ such that $\mu_{n-1} = \mu_{n} =0$ and let $M$ be polynomial of degree $d = |\mu|$. Then $$[M  :  L_n(\mu)] = [M^{U_{I}^{+}} : L_{n-2}(\overline{\overline{\mu}})].$$ 
\end{lem}

\begin{proof}
 Let $0 = M_{(0)} < \dots < M_{(t)} = M$ be a composition series for $M$. We will prove inductively that $[M_{(i)}  :  L_n(\mu)] = [M_{(i)}^{U_{I}^{+}} : L_{n-2}(\overline{\overline{\mu}})]$ for $1 \leq i \leq t$. Suppose $$[M_{(i-1)} : L_n(\mu)] = [M_{(i-1)}^{U_{I}^{+}} : L_{n-2}(\overline{\overline{\mu}})]$$ and let $\nu \in X^{+}(n)$ be the dominant weight such that $M_{(i)} / M_{(i-1)} \cong L_n(\nu)$. We have a short exact sequence $$0 \rightarrow M_{(i-1)}^{U_{I}^{+}} \rightarrow M_{(i)}^{U_{I}^{+}} \rightarrow (M_{(i)}/M_{(i-1)})^{U_{I}^{+}}$$ where $(M_{(i)}/M_{(i-1)})^{U_{I}^{+}} \cong L_n(\nu)^{U_{I}^{+}}$ which by Proposition \ref{janprop} is isomorphic to the $\GL(n-2)\times T_2$ module with highest weight $\nu$. 
 
 Note that $\mu = \nu$ if and only if $\overline{\overline{\mu}} = \overline{\overline{\nu}}$ since $|\nu| = |\mu| = |\overline{\overline{\mu}}|$. 
 
Therefore, restricting to $\GL(n-2)$,  if $\mu \neq \nu$ then $\overline{\overline{\mu}} \neq \overline{\overline{\nu}}$ and we have $$ [M_{(i)}^{U_{I}^{+}} : L_{n-2 }(\overline{\overline{\mu}})] = [M_{(i-1)}^{U_{I}^{+}} : L_{n-2 }(\overline{\overline{\mu}})]  =  [M_{(i-1)} : L_{n}(\mu)] = [M_{(i)} : L_{n}(\mu)]. $$ 
 
 If $\mu = \nu$ then we will show that the map $$M_{(i)}^{U_{I}^{+}} \rightarrow (M_{(i)}/M_{(i-1)})^{U_{I}^{+}}$$ is surjective. Let $v + M_{(i-1)} \in (M_{(i)}/M_{(i-1)})^{U_{I}^{+}} \cong L(\mu)^{U_{I}^{+}} \cong \bigoplus_{\eta \in \Z {S \backslash I}} L(\mu)_{\mu-\eta}$. Then $v$ can be chosen to be a sum of weight vectors $x$ of weight $\mu-\eta$ for some $\eta \in \Z {I}$. Note that $I$ does not contain $\epsilon_{n-2} - \epsilon_{n-1}$ or $\epsilon_{n-1} - \epsilon_n$ and therefore since $\mu_{n-1} = \mu_{n} =0$ we have $(\mu-\eta)_{n-1} = (\mu-\eta)_{n} = 0$. Therefore for each weight vector $x$ of weight $\mu - \eta$ we must have $E_{an}^{(k)} x= 0$ for $k>0$ and $1 \leq a \leq n-1$ since the weight of $E_{an}^{(k)} x$ has a negative entry and all of the weights of $M$ have non-negative entries by assumption that $M$ is polynomial. Similarly $E_{an-1}^{(k)} x = 0$ for $k>0$ and $1 \leq a < n-1$. Therefore $E_{an}^{(k)} v = \sum_{x} E_{an}^{(k)}  x =0$ and hence $v \in M_{(i)}^{U_{I}^{+}}$.
 
 Therefore  \begin{align*} [M_{(i)}^{U_{I}^{+}} : L_{n-2 }(\overline{\overline{\mu}})] &= [M_{(i-1)}^{U_{I}^{+}} : L_{n-2 }(\overline{\overline{\mu}})] + 1 \\
 &= [M_{(i-1)} : L_n(\mu)] + [ L_n(\nu) : L_n(\mu) ] \\
 &=  [M_{(i)} : L_n(\mu)].
 \end{align*} This completes the inductive step and hence the statement is true for all $0 \leq i \leq t$ in particular $$[M  :  L_n(\mu)] = [M^{U_{I}^{+}} : L_{n-2}(\overline{\overline{\mu}})].$$ 
\end{proof}

Recall we are interested in relating composition factors of $M \otimes \bigwedge\nolimits^2 V_n^*$ where $M$ is a simple $\GL(n)$-module to composition factors of $M$ restricted to $\GL(n-2)$. In particular we want to apply the previous lemma to the case where $M= L(\lm) \otimes \bigwedge\nolimits^2 V_n^*$. However this will not be polynomial and so we cannot do this directly . Instead we will show that the composition factors we are interested in show up in the \say{polynomial part} of $L(\lm) \otimes \bigwedge\nolimits^2 V_n^*$. To do that we need to describe a certain filtration of $L(\lm) \otimes \bigwedge\nolimits^2 V_n^*$. 

Recall the definitions of the hyperalgebra $U(n, \F)$ and subalgebras $U^{-}(n)$,  $U^{+}(n)$ and $U^{0}(n)$ from Section \ref{hypgln}. Recall further that $U(n, \F) = U^{-}(n)U^{0}(n)U^{+}(n)$ and therefore to deduce some subspace $M$ of a $U(n, \F)$-module is a $U(n, \F)$-submodule it is enough to show it is invariant under each of the subalgebras $U^{-}(n)$,  $U^{+}(n)$ and $U^{0}(n)$.

\begin{lem}\label{BKlem}
Let $\lm \in X^{+}(n)$. There exists a filtration $0 \leq N_{(1)} \leq \dots \leq N_{(t)} = L_n(\lm) \otimes \bigwedge\nolimits^2 V_n^*$ such that each $N_{(i)}/N_{(i-1)}$ is a possibly zero quotient of $\Delta_n(\lm- \epsilon_x - \epsilon_y)$ for $(x,y)$ a $\lm$-removable pair. 
\end{lem}

\begin{proof} 

We will prove the statement of the lemma by first constructing a filtration for $\Delta_n(\lm) \otimes  \bigwedge\nolimits^2 V_n^*$ and then considering the image under the map induced by the surjection $\Delta_n(\lm) \rightarrow L_n(\lm)$. 

We begin by defining an ordering on pairs $(i,j)$ for $1\leq i < j \leq n$. Let $r_1:=(n-1,n)$ and for $r_a = (i,j)$ define $$r_{a+1} = \begin{cases} (i-1,j) \hspace{1cm} & i \neq 1 \\ (j-2, j-1) \hspace{1cm} & i = 1 \end{cases}.$$ If $r_a=(i,j)$ let $X_{r_a} = X_{i} \wedge X_j$ and let $s:=  {n\choose2}$ so that $r_s = (1,2)$. 

Let $w_{\lm}$ denote a weight vector in $\Delta_n(\lm)$ of weight $\lm$ so that $w_{\lm}$ generates $\Delta_n(\lm)$ as a $U^{-}(n)$-module. This means that for any $w \in \Delta_n(\lm)$ we have $w=f w_{\lm}$ for some $f \in U^{-}(n)$. We now show by a decreasing induction on $1 \leq a \leq s$ that $w \otimes X_{r_a}$ lies in the $U^{-}(n)$-module generated by $w_{\lm} \otimes X_{r_a}, \dots w_{\lm} \otimes X_{r_s}$. We have $$f(w_{\lm} \otimes X_{r_a}) = w \otimes X_{r_a} + q$$ where $q$ is a linear combination of elements $w' \otimes X_{r_b}$ with $w' \in \Delta_n(\lm)$ and $b > a$. By induction these elements lie in the $U^{-}(n)$-module generated by $w_{\lm} \otimes X_{r_{a+1}}, \dots, w_{\lm} \otimes X_{r_s}$ hence the inductive hypothesis is satisfied. Therefore  $\Delta_n(\lm) \otimes \bigwedge\nolimits^2 V_n^*$ is generated as a $U^{-}(n)$-module by $\{w_{\lm} \otimes X_{r_a} | 1 \leq a \leq s\}$. 

Now let $s_1, \dots, s_t$ be the pairs which are removable for $\lm$ ordered by their place in the ordering above (so $s_1 = r_1 = (n-1,n)$).  Suppose a pair $r_a:=(i,j)$ is not removable for $\lm$. Then there is a $\lm$-removable pair $s_b:= (x,y)$ such that $\lm_{i} = \lm_{x}$ and $\lm_{j} = \lm_{y}$. If $i \neq x$ then $F_{ix} w_{\lm} \in \Delta_n(\lm)_{\mu}$ where $\mu = \lm - (\epsilon_i - \epsilon_x)$. Furthermore $\Delta_n(\lm)_{\mu} \neq 0$ only if $\Delta_n(\lm)_{w \mu} \neq 0$ for each $w \in W_n$. Consider the element $\sigma_{i,x} \in W_n$ which swaps the $i^{th}$ and $x^{th}$ entries of a weight $\mu$. We have \begin{align*} \lm - \sigma_{i,x} \mu &= (0, \dots, \lm_i-\lm_{x}+1, 0, \dots, \lm_{x}-\lm_i -1, \dots, 0) \\ &=-(\epsilon_i - \epsilon_{x}). \end{align*} Hence $\sigma_{i,{x}} \mu$ is not subdominant to $\lm$. Therefore $\Delta_n(\lm)_{\mu} = 0$ and $F_{ix} w_{\lm}=0$. The same argument shows that if $j \neq y$ then $F_{jy} w_{\lm}=0$. Therefore $$F_{ix}F_{jy} ( w_{\lm} \otimes X_{s_b}) = w_{\lm} \otimes X_{i} \wedge X_{j}$$ where, for brevity, $F_{ab}$ is interpreted as $1$ if $a=b$. It follows that $\Delta_n(\lm) \otimes \bigwedge\nolimits^2 V_n^*$ is generated as a $U^{-}(n)$-module by $\{w_{\lm} \otimes X_{s_b} | 1 \leq b \leq t\}$. 

Let $M_{(0)} = 0$ and for $i \geq 1$ define $M_{(i)}$ be the $U^{-}(n)$-submodule of  $M:= \Delta_n(\lm) \otimes \bigwedge\nolimits^2 V_n^*$ generated by $M_{(i-1)}$ and $w_{\lm} \otimes X_{s_i}$. We will now show by induction that $M_{(i)}$ is a $\GL(n)$-module giving us a filtration of $M$. 

Note that since $M_{(i)}$ is generated as a $U^{-}(n)$-module by weight vectors it is $U^{(0)}$-invariant. We wish to show $M_{(i)}$ is a $U(n, \F)$-module and therefore a $\GL(n)$-module. It remains to show $M_{(i)}$ is a $U^{+}(n)$-module. Recall a vector $v + M_{(i-1)} \in M_{(i)}/M_{(i-1)}$ is primitive if and only if $E_{ij}^{(k)} v \in M_{(i-1)}$ for all $1 \leq i <j \leq n$ and $k>0$. Since we assume $M_{(i-1)}$ is  $U^{+}(n)$-invariant, if $M_{(i)}/M_{(i-1)}$ is generated by primitive vectors then it is a $U^{+}(n)$-invariant. Therefore it is enough to show that $w_{\lm} \otimes X_{{s_i}} + M_{(i-1)}$ is primitive in $M_{(i)} / M_{(i-1)}$.

We may assume $i>0$ and $M_{(i-1)} \neq M$ since otherwise the statement is obvious.
We may also assume $w_{\lm}\otimes X_{s_i} + M_{(i-1)}$ is nonzero since the zero vector is always primitive.  If $s_i = (x,y)$ then $$E_{ab}^{(k)} (w_{\lm}\otimes X_{s_i} )= \begin{cases} -w_{\lm}\otimes X_{b} \wedge X_{y}  \hspace{1cm} & k=1, a=x \\  -w_{\lm}\otimes X_{x} \wedge X_{b}  \hspace{1cm} & k=1, a=y \\ 0  \hspace{1cm} & \text{otherwise} \end{cases}$$ for $k>0$.
In the first option above $b>x$ and in the second $b>y$ so in all cases $E_{ab}^{k} (w_{\lm}\otimes X_{s_i} ) \in M_{(i-1)}$ and hence $w_{\lm}\otimes X_{s_i} + M_{(i-1)}$  is primitive in $M_{(i)}/M_{(i-1)}$. This completes our inductive argument but also implies $M_{(i)}/M_{(i-1)}$ is a quotient of $\Delta_n(\lm- \epsilon_x - \epsilon_y)$ where $s_i = (x, y)$ (see Lemma \ref{prim}).

Finally let $0= \overline{M_{(0)}} \leq\overline{M_{(1)}} \leq \dots \leq L_n(\lm) \otimes \bigwedge\nolimits^2 V_n^*$ be the image of the above filtration under the map $\Delta_n(\lm) \otimes \bigwedge\nolimits^2 V_n^* \rightarrow L_n(\lm) \otimes \bigwedge\nolimits^2 V_n^*$. Note that $\overline{M_{(i)}}$ may be equal to  $\overline{M_{(i-1)}}$ but $\overline{M_{(i)}}/\overline{M_{(i-1)}}$ will be a (possibly zero) quotient of $\Delta_n(\lm- \epsilon_x - \epsilon_y)$. 

\end{proof}

Let us now prove Theorem \ref{key}.

 \begin{proof}[Proof of Theorem \ref{key}]
 If $N:=L_n(\lm) \otimes \bigwedge\nolimits^2 V_n^*$ then by the linkage principle we can write $N = N' \oplus N''$ where $N'$ is a sum of weight spaces $N_{\nu}$ for weights $\nu$ such that $\nu \sim \lm - \epsilon_i - \epsilon_n$ where $1 \leq i \leq n-1$ and $(i,n)$ is a $\lm$-removable pair.
 
Since $\mu$ is not a weight of $N'$ we have \begin{equation}\label{one} [N: L_n(\mu)]= [N': L_n(\mu)]+ [N'': L_n(\mu)] = [N'': L_n(\mu)]. \end{equation} We will now show that \begin{equation}\label{two} [N^{U_{I}^{+}}: L_{n-2}(\overline{\overline{\mu}})] = [N''^{U_{I}^{+}} : L_{n-2}(\overline{\overline{\mu}})]. \end{equation} 

We have \begin{equation}\label{three} [N^{U_{I}^{+}} : L_{n-2}(\overline{\overline{\mu}})] = [N''^{U_{I}^{+}} : L_{n-2}(\overline{\overline{\mu}})] + [N'^{U_{I}^{+}} : L_{n-2}(\overline{\overline{\mu}})] \end{equation} and we will show that $[N'^{U_{I}^{+}} : L_{n-2}(\overline{\overline{\mu}})] =0$. We will do this by showing that for each composition factor $L_n(\nu)$ of $N'$ we have $$[L_n(\nu)^{U_{I}^{+}} : L_{n-2}(\overline{\overline{\mu}})]=0.$$ Suppose $L_{n-2}(\overline{\overline{\mu}})$ is a composition factor of $L_n(\nu)^{U_{I}^{+}}$. Then by Proposition \ref{janprop} we have that $\overline{\overline{\mu}}= \overline{\overline{\nu}}$. By our assumptions on $\mu$ this implies $\nu_{n-2} = \mu_{n-2} = 0$ and $|\overline{\overline{\nu}}| = |\overline{\overline{\mu}}| = |\lambda | -2$. 

Recall that the weight $\nu$ of $N'$ must be linked to some $\lm-\epsilon_i -\epsilon_j$ for $1 \leq i \leq n-1$ and $j =n$. Therefore $|\nu| = \sum_{\alpha \in \Z/p\Z} cont_{\alpha}(\lm-\epsilon_i-\epsilon_j) = |\lm - \epsilon_i - \epsilon_j| = |\lm| -2$.

Note this implies $|\nu| = |\overline{\overline{\nu}}|$ and hence that $\nu_{n-1} + \nu_{n} =0$. Combining this with the fact that $\nu_{n-2}=0$ and $\nu \in X^{+}(n)$ and hence $0 = \nu_{n-2} \geq \nu_{n-1} \geq \nu_{n}$ we may deduce that $\nu_{n-1} = \nu_{n} = 0$ and hence $\mu = \nu$. This is a contradiction as $\mu$ is not a weight of $N'$. Therefore $[N'^{U_{I}^{+}} : L_{n-2}(\overline{\overline{\mu}})] =0$ and (\ref{two}) follows from $(\ref{three})$.

We will now show that \begin{equation}\label{four}  [N^{U_{I}^{+}} : L_{n-2}(\overline{\overline{\mu}})] = [ L_n(\lm)^{(1,1) \times} : L_{n-2}(\overline{\overline{\mu}})]. \end{equation}

It follows from Proposition \ref{iso} that $$[N^{U_{I}^{+}}: L_{n-2}(\overline{\overline{\mu}})] = [L_n(\lm)^{(0,0)} \oplus L_n(\lm)^{(1,0)} \oplus L_n(\lm)^{(0,1) \times} \oplus L_n(\lm)^{(1,1) \times} : L_{n-2}(\overline{\overline{\mu}})].$$ Furthermore the weights $\eta$ corresponding to nonzero weight spaces of $L_n(\lm)$ are subdominant to $\lm$ and therefore can be derived from $\lm$ by subtracting linear combinations of the simple roots $\epsilon_i - \epsilon_{i+1}$ which does not affect $|\eta|$ and hence $|\eta| = |\lm|$. Recall that $$L_n(\lm)^{(0,0)} = \sum_{\nu_{n-1} = \nu_n = 0 \\ |\nu|=|\lm|} L_n(\lm)_{\nu}$$ so as a $\GL(n-2)$-module its only weights are of the form $\eta$ where$|\eta| = |\lm|$. Therefore $\overline{\overline{\mu}}$ is not a weight of $L_n(\lm)^{(0,0)}$. Similarly one can see $\overline{\overline{\mu}}$ is not a weight of $L_n(\lm)^{(1,0)}$ or $L_n(\lm)^{(0,1)}$ and hence we have shown equation(\ref{four}). 

It follows from (\ref{one}) and (\ref{two}) that to show that $$[N: L_n(\mu)] = [N^{U_{I}^{+}} : L_{n-2}(\overline{\overline{\mu}})]$$ it is enough to show  \begin{equation}\label{five} [N'': L_n(\mu)] = [N''^{U_{I}^{+}} : L_{n-2}(\overline{\overline{\mu}})] .\end{equation} Therefore, to complete the proof it remains only to show (\ref{five}). If we can show $N''$ is polynomial of degree $|\lm|-2$ then (\ref{five}) follows from Lemma \ref{poly1}.

If $(x,y)$ is a $\lm$-removable pair with $1 \leq x <y <n-1$ then $\Delta(\lm - \epsilon_x - \epsilon_y)$ is polynomial of degree $|\lm|-2$ since $\lm_n=0$. Recall the ordering $s_1, \dots, s_t$ of removable pairs for $\lm$ defined in the proof of Lemma \ref{BKlem}. Recall also from Lemma \ref{BKlem} that we have a filtration  $$0 \leq N_{(1)} \leq \dots \leq N_{(t)} = L_n(\lm) \otimes \bigwedge\nolimits^2 V_n^*$$ where $N_{(i)}/N_{(i-1)}$ is a possibly zero quotient of $\Delta_n(\lm- \epsilon_{x_i} - \epsilon_{y_i})$ where $s_i = (x_i,y_i)$. Let $1 \leq a \leq t$ be maximal so that $s_a = (x,y)$ with $y \in \{n, n-1\}$. Note that since $\lm_n =0$ we can see $n-1$ is not $\lm$-removable and it is only $\lm - \epsilon_j$ removable if $j=n$. Therefore the assumption that $y \in \{n, n-1\}$ is equivalent to $y=n$. 
 We will show by induction that for $i \leq a$ we have $N_{(i)} \leq N'$. Suppose $N_{(i-1)} \leq N'$ and consider the filtration $$0 \leq N_{(i)}/N_{(i-1)} \leq \dots \leq N/ N_{(i-1)} \cong N'/N_{(i-1)} \oplus N''.$$ Since $N_{(i)}/N_{(i-1)}$ is a quotient of $\Delta(\lm - \epsilon_{x_i} - \epsilon_{n})$ it contains only weights linked to $\lm - \epsilon_{x_i} - \epsilon_{n}$ and therefore it is contained in $N'/ N_{(i-1)}$. Hence $N_{(i)} \leq N'$ and the inductive step is complete. In particular $N_{(a)} \leq N'$ and therefore $N''$ is a quotient of $N/ N_{(a)}$ which is polynomial of degree $|\lm|-2$ as it has a filtration by quotients of $\Delta(\lm - \epsilon_x - \epsilon_y)$ for $1 \leq x < y <n-1$. Therefore $N''$ is polynomial and the proof is complete. 
 \end{proof}

It is worth noting that the condition that  $\mu \nsim \lm - \epsilon_i - \epsilon_n$ is necessary for the statement of the result to hold. In particular we were able to find a counterexample when $\mu \sim \lm - \epsilon_i - \epsilon_n$. Specifically, taking $n=5$ and $p=7$ with $\lm = (4,0,0,0,0)$ and $\mu = (2,0,0,0,0)$, one can check using MAGMA that $$[L_n(\lm) \otimes \bigwedge\nolimits^2 V_n^* : L_n(\mu)] = [L_n(\lm)^{(1,1)} : L_{n-2}(\overline{\overline{\mu}})] \neq [L_n(\lm)^{(1,1) \times} : L_{n-2}(\overline{\overline{\mu}})].$$

\section{\normalsize{SOME COMPOSITION MULTIPLICITIES FOR $L_n(\lm) \otimes \bigwedge\nolimits^2 V_n^*$}}\label{mainsec}

In this section we will motivate why Theorem \ref{key} is useful for understanding composition factors of tensor products. In particular we will describe a method for determining the composition multiplicities for some of the composition factors of the $\GL(n-2)$-module $L_n(\lm)^{(1,1) \times}$. We can then use Theorem \ref{key} to turn these into results about the tensor product $L_n(\lm) \otimes \bigwedge\nolimits^2 V_n^*$. 

Throughout this section we will consider the case where $\lm = (\lm_1,\dots,\lm_n) \in X^{+}(n)$ is such that $\lm_n = \lm_{n-1} = 0$. Let $\lm_s$ be the last nonzero entry in $\lm$. Finally let $\alpha(a,b) = \sum_{r=a}^{b-1} \alpha_r$ and define $\mu_i := \lm - \alpha(i,n-1) - \alpha(s, n)$ and $\omega_i := \lm - \alpha(i,n-1) - \alpha(s,n-1)$. Our main goal for this section is to describe a method for computing the composition multiplicities for composition factors $L_{n-2}(\omi)$ of the $\GL(n-2)$-module $L_n(\lm)^{(1,1) \times}$. 

Let us first describe  $\dim((L_n(\lm)^{(1,1) \times})_{\omi})$ for $1 \leq i \leq s-1$. The following results reduce this to understanding the weight spaces of the $\GL(n)$-module $L_n(\lm)$. 

\begin{lem}\label{munu} 
Let $\mu, \nu \in X(n)$ be the weights of two nonzero vectors in $L_n(\lm)^{(1,1)}$. Then $\mu = \nu$ if and only if $\overline{\overline{\mu}} = \overline{\overline{\nu}}$. 
\end{lem}

\begin{proof}
If $\mu = \nu$ then $\overline{\overline{\mu}} = \overline{\overline{\nu}}$ by definition. Suppose $\overline{\overline{\mu}} = \overline{\overline{\nu}}$. Let $\mu = (\mu_1, \dots, \mu_n)$ and $\nu = (\nu_1, \dots, \nu_n)$. Since they are weights of vectors in $L_n(\lm)^{(1,1)}$ we can write $\nu =  \lm - \alpha_{n-1} - 2 \alpha_{n-2} - \sum_{i=1}^{n-3} n_i \alpha_i$ and $\mu =  \lm - \alpha_{n-1} - 2 \alpha_{n-2} - \sum_{i=1}^{n-3} m_i \alpha_i$. In which case $\overline{\overline{\mu}} = \lm - \sum_{i=1}^{n-3} m_i \alpha_i - 2\epsilon_{n-2}$ and $\overline{\overline{\nu}} =\lm - \sum_{i=1}^{n-3} n_i \alpha_i - 2\epsilon_{n-2}$. Since $\overline{\overline{\mu}} = \overline{\overline{\nu}}$ this implies $m_i=n_i$ for $ 1 \leq i \leq n-3$ and hence $\mu = \nu$. 
\end{proof}

\begin{thm}\label{weightsspaces} Let $\lm \in X^{+}(n)$ be such that $\lm_{n} = \lm_{n-1} = 0$ and suppose $p \neq 2$. Let $\mu \in X(n)$ be such that $L_{n}(\lm)_{\mu} \subseteq L_n(\lm)^{(1,1)}$. Finally let $\omega$ be the weight $(\mu_1,\dots, \mu_{n-2},2,0) \in X(n)$. Then, $$ \dim((L_n(\lm)^{(1,1) \times})_{\overline{\overline{\mu}}}) = \dim(L_n(\lm)_{\mu}) - \dim(L_n(\lm)_{\omega}).$$ \end{thm}

\begin{proof}
Let us first show that $\dim((L_n(\lm)^{(1,1) \times})_{\overline{\overline{\mu}}}) = \dim (L_n(\lm)^{(1,1)\times} \cap L_n(\lm)_{\mu}).$ Indeed there is an obvious inequality coming from the fact that if $0 \neq v \in L_n(\lm)^{(1,1) \times}$ has $\GL(n)$-weight $\nu$ then it has $\GL(n-2)$-weight $\overline{\overline{\nu}}$. Furthermore, Lemma \ref{munu} implies that if $0 \neq v \in L_n(\lm)^{(1,1) \times}$ has $\GL(n)$-weight $\eta$ where $\overline{\overline{\eta}} = \overline{\overline{\nu}}$ then $\nu = \eta$ so $v \in L_n(\lm)_{\nu}$. 

We will now show that $\dim (L_n(\lm)^{(1,1)\times} \cap L_n(\lm)_{\mu}) = \dim(L_n(\lm)_{\mu}) - \dim(L_n(\lm)_{\omega}).$ Consider the map $E_{n-1n}: L_n(\lm)_{\mu} \rightarrow L_n(\lm)_{\omega}$. Clearly $L_n(\lm)^{(1,1) \times} \cap L_n(\lm)_{\mu}$ is the kernel of this map. We will show it is surjective from which the statement follows. We also have a map $F_{n-1n}:L_n(\lm))_{\omega} \rightarrow L_n(\lm)_{\mu}$. If we can show the map $E_{n-1n} F_{n-1n}$  is an isomorphism then it follows that $E_{n-1n}$ is surjective. Therefore it is enough to show that if $v_1, \dots, v_s$ is a basis of $L_n(\lm)_{\omega}$ then $E_{n-1n}F_{n-1n} v_1, \dots, E_{n-1n}F_{n-1n} v_s$ is also a basis of $L_n(\lm)_{\omega}$. Note that since $\lm_n =0$ and $\omega_{n} =0$ it follows from Lemma \ref{neg} that $E_{n-1n} L_n(\lm)_{\omega} =0$. Recalling the relations in Lemma \ref{relhyp} and the action of $H_i$ discussed in Section \ref{hypgln} we see that \begin{align*} E_{n-1n} F_{n-1n} v_i - F_{n-1n}E_{n-1n} v_i &= H_{n-1} v_i - H_{n} v_i \\ &= 2 v_i.\end{align*} Therefore $E_{n-1n}F_{n-1n} v_i = 2 v_i$ and clearly $2v_1, \dots, 2v_s$ is a basis of $L_n(\lm)_{\omega}$. 
\end{proof}

Let us now describe $\dim(L_n(\lm)_{\mu_i})$ and $\dim(L_n(\lm)_{\omega_i})$ for $1 \leq i \leq s-1$. Observe that $\mu_i$ and $\omega_i$ are in the same Weyl orbit as $\mu^i  := \lm - \alpha(i,s+1) - \alpha(s,s+2)$ and $\omega^i := \lm - \alpha(i,s+1) - \alpha(s,s+1)$ respectively. Therefore it is enough to describe $\dim(L_n(\lm)_{\mu^i})$ and $\dim(L_n(\lm)_{\omega^i})$. To do so we will first describe the dimension of the weight space $L_n(\lm)_{\lm -\alpha(i,j)}$ for $1 \leq i < j \leq s+1$. When $\lm_{s} =1$ we note that  $\dim(L_n(\lm)_{\mu^i}) = \dim(L_n(\lm)_{\lm -\alpha(i,s+1)})$ since the defining weights are in the same Weyl orbit. 

For $1 \leq i \leq k < j$, let $X_{j-1}^{i,j}(k) = \lm_k - \lm_j $ and $Y_{j-1}^{i,j}(k) = - (\lm_k - \lm_j)$. For $i \leq u  < j-1$ let $u_{i,j}^* = j-1 -u +i$ and define $$X_u^{i,j}(k) = \begin{pmatrix} X_{u+1}^{i,j}(k) & Y_{u+1}^{i,j}(u_{i,j}^*) \\ X_{u+1}^{i,j}(k) + Y_{u+1}^{i,j}(u_{i,j}^*) & X^{i,j}_{j-1}(u_{i,j}^*) (X_{u+1}^{i,j}(k) + Y_{u+1}^{i,j}(u_{i,j}^*) )\end{pmatrix}$$ where $$Y_u^{i,j}(k) = \begin{pmatrix} Y_{u+1}^{i,j}(k) & -Y_{u+1}^{i,j}(k) \\ 0 & Y^{i,j}_{j-1}(k) (X_{u+1}^{i,j}(i) + Y_{u+1}^{i,j}(u_{i,j}^*)) \end{pmatrix}.$$ For brevity let $X^{i}_u(k) := X^{i,s+1}_u(k)$ and $Y^{i}_u(k) := Y^{i,s+1}_u(k)$. Let $$M^{i} = \begin{pmatrix} \lm_s X_{i+1}^i(i) & Y_{i+1}^i(s) &(\lm_s -1) Y_{i+1}^i(s) \\ Y_{i+1}^i(s) & \lm_s X_{i+1}^i(i) & (\lm_s -1) Y_{i+1}^i(s) \\ \lm_s(X_{i+1}^i(i) + Y_{i+1}^i(s)) & \lm_s(X_{i+1}^i(i) + Y_{i+1}^i(s)) &\lm_s (\lm_s -1) (X_{i+1}^i(i) + Y_{i+1}^i(s)) \end{pmatrix} $$ and finally let $$W^{i} = \begin{pmatrix} \lm_s X_{i+1}^{i}(i) + Y_{i+1}^{i}(s) & (\lm_s -1) Y_{i+1}^{i}(s) \\ \lm_s(X_{i+1}^{i}(i) + Y_{i+1}^{i}(s) ) & \frac{1}{2} \lm_s (\lm_s -1) (X_{i+1}^{i}(i) + Y_{i+1}^{i}(s) ) \end{pmatrix}.$$

From now on define $x_{ij} := \rank(X_{i}^{i,j}(i) \mod p)$, $m_s^i : = \rank( M^i \mod p)$ and $w_s^i :=\rank(W^i \mod p)$.

\begin{thm}\label{rankdims}
For $\lm \in X^{+}(n)$ such that $\lm_s$ the last nonzero entry in $\lm$ and $s < n-1$ we have, \begin{enumerate} \setlength\itemsep{-1em} \item $\dim(L_n(\lm)_{\lm - \alpha(i,j)}) = x_{i,j}$, \\ \item $\dim(L_n(\lm)_{\mu^i}) = m_s^i$ and \\ \item $\dim(L_n(\lm)_{\omega^i}) = w_s^i$. \end{enumerate} 
\end{thm}

\begin{proof}
Recall that the hyperalgebra $U(n,\F)$ is spanned by the symbols $E_{ij}^{(k)}, F_{ij}^{(k)}$ and $H_i^{(k)}$ for $1 \leq i < j \leq n$ and $k>0$. Let us fix an order on the pairs $(i,j)$ defining the symbols $E_{ij}^{(k)}$ and $F_{ij}^{(k)}$. We say $(i,j)$ is less than $(a,b)$ if either $j<b$ or $j=b$ and $i<a$. The subalgebras $U^+(n)$ and $U^-(n)$ can be generated by ordered symbols given by tuples $\underline{x} = (k_1,\dots,k_m)$ where $E_{\underline{x}} = E_{12}^{(k_1)} \dots E_{n-1n}^{(k_m)}$ and $F_{\underline{x}} = F_{12}^{(k_1)} \dots F_{n-1n}^{(k_m)}.$

Let $v^{+}$ be a weight vector for $L_n(\lm)$ of weight $\lm$. For a weight $\mu$ for $L_n(\lm)$ let $I_{\mu}(\lm)$ denote the set of elements $F_{\underline{x}}$ of $U^{-}(n)$ such that $F_{\underline{x}} v^{+} \in L_n(\lm)_{\mu}$. For $F_{\underline{x}}, F_{\underline{y}} \in I_{\mu}(\lm)$ it follows from Kostant's Theorem that $E_{\underline{x}} F_{\underline{y}} v^+ = n_{\underline{x}\underline{y}} v^{+}$ where $n_{\underline{x}\underline{y}}$ is an integer. The dimension of $L_n(\lm)_{\mu}$ is given by the $\mod p$ rank of the matrix $X_{\mu} = (n_{\underline{x}\underline{y}})_{\underline{x}\underline{y}}$ ranging over all possible $F_{\underline{x}}, F_{\underline{y}} \in I_{\mu}(\lm)$ (see for example Chapter 26 and 27 in \cite{HUM1972}). We therefore need to describe the set $I_{\mu}(\lm)$ and the matrix $X_{\mu}$ for $\mu = \lm - \alpha(i, j)$,  $\lm - \alpha(i,s+1)-\alpha(s,s+2)$ and $\lm - \alpha(i,s+1) - \alpha(s,s+1)$. 

Let us first look at the case where $\mu = \lm - \alpha(i, j)$ . The possible $F_{\underline{x}} \in I_{\mu}(\lm)$ must be of the form $F_{is_1}F_{s_1 s_2} \dots F_{s_r j}$ where $S = \{s_1,\dots,s_r\}$ is an increasing subset of the interval $(i,j)$. Let us describe an ordering of the subsets of $(i,j)$ and therefore the elements of $I_{\mu}(\lm)$. For $S= \{s_1,\dots,s_r\}$ and $T = \{t_1,\dots, t_m\}$ distinct subsets of $(i,j)$ we say $S$ is below $T$ in the ordering if $S = \emptyset$, $t_m > s_r$ or $t_m = s_r$ and $S \ \{s_r\}$ is below $T \ \{t_r\}$. For example, the ordered sequence of subsets of $(s-2,s+1)$ would be $\emptyset, \{s-1\} , \{s\} \{s-1,s\}$. Let us identify $I_{\mu}(\lm)$ with the set of subsets of $(i,j)$ ordered in this way.

Let $N^{i,j}:=(n_{ST})_{S,T \in I_{\mu}(\lm)}$ be the matrix whose entries are the integers $n_{ST}$ such that $n_{ST} v^{+} = E_{is_1}E_{s_1s_2} \dots E_{s_r j} F_{i t_1}F_{t_1 t_1} \dots F_{t_m j} v^{+}$ for all $S,T \in I_{\mu}(\lm)$. The sets indexing the rows and columns are ordered as above. We will show inductively that $N^{i,j}= X^{(i,j)}_{i}(i)$.

For $i \leq  l \leq j-1$, let $I_l$ be the set of subsets of $(i,l]$. Further, let $I_l^{-}$ be the subset of $I_l$ consisting of only sets which contain $l$ and define $I_l^{+} := I_l / I_{l}^{-}$.

Let $N^{i,j}_{l}$ be the submatrix of $N^{i,j}$ containing the rows and columns indexed by the set $I_l$. We will prove by induction on $i< l \leq j-1$ that $N^{i,j}_{l} = X_{i+j-(l+1)}^i(i)$. Note that $E_{ij}F_{ij} v^{+} = (\lm_{i} - \lm_j ) v^{+}$ and hence $N^{i,j}_{i} = \lm_i - \lm_j = X^{i,j}_{j-1}(i)$. For the base case of our induction we will show $$N^{i,j}_{i+1} = \begin{pmatrix} X_{j-1}^{i,j}(i) & Y_{j-1}^{i,j}(i+1) \\ X_{j-1}^{i,j}(i)+Y_{j-1}^{i,j}(i+1) & X_{j-1}^{i,j}(i+1) (X_{j-1}^{i,j}(i)+Y_{j-1}^{i,j}(i+1) ) \end{pmatrix}$$ where $Y^{i,j}_{j-1}(i+1) = -(\lm_{i+1}-\lm_{j})$.

We have, \begin{align*} E_{ij}F_{i i+1}F_{i+1 j} v^{+} &= (F_{i i+1} E_{i j} - E_{i+1 j} ) F_{i+1 j} v^{+}  \\ & = -(\lm_{i+1} - \lm_j) v^{+},  \end{align*} \begin{align*} E_{i i+1}E_{i+1 j}F_{ij} v^{+} & = E_{i i+1} (F_{i j} E_{i+1 j} + F_{i i+1} ) v^{+} \\ &= (\lm_i - \lm_{i+1}) v^{+} \end{align*} and finally \begin{align*} E_{i i+1}E_{i+1 j}F_{i i+1}F_{i+1 j} v^{+} & =E_{i i+1}F_{i i+1}E_{i+1 j}F_{i+1 j} v^{+} \\ &= (\lm_{i+1} - \lm_j)( \lm_i - \lm_{i+1}) v^{+}. \end{align*}

Recall that for $i \leq u <j-1$, we denote $u_{ij}^* = j -1 +i - u$. Now let us assume $N^{i,j}_{l-k}  = X^{i,j}_{(l-k)_{i,j}^*}(i)$ for $1 \leq k \leq l-i$. In particular we assume that $(n_{ST})_{S,T \in I_{l-1}} = X^{i,j}_{(l-1)_{i,j}^*}(i)$ and hence $(n_{ST})_{S \in I_{l-1}^{+},T \in I_{l-1}^{-}} = Y^{i,j}_{(l-2)_{i,j}^*}(l-1).$  We will show \begin{enumerate} \setlength\itemsep{-1em} \item  $(n_{ST})_{S, T \in I_{l}^{+}} = X^{i,j}_{(l-1)_{i,j}^*}(i),$ \\ \item  $(n_{ST})_{S \in I_{l}^{+}, T \in I_{l}^{-}}  = Y^{i,j}_{(l-1)_{i,j}^*}(l),$ \\ \item $(n_{ST})_{S \in I_{l}^{-}, T \in I_{l}^{+}} = X^{i,j}_{(l-1)_{i,j}^*}(i) +  Y^{i,j}_{(l-1)_{i,j}^*}(l),$ \\ \item  $(n_{ST})_{S \in I_{l}^{-}, T \in I_{l}^{-}}  = \lm_{l} (X^{i,j}_{(l-1)_{i,j}^*}(i) +  Y^{i,j}_{(l-1)_{i,j}^*}(l))$ \end{enumerate} where $$Y^{i,j}_{(l-1)_{i,j}^*}(l) = \begin{pmatrix} Y^{i,j}_{(l-2)_{i,j}^*}(l) & - Y^{i,j}_{(l-2)_{i,j}^*}(l) \\ 0 & -(\lm_l-\lm_j) (X_{(l-2)_{i,j}^*}^{i,j}(i) + Y^{i,j}_{(l-2)_{i,j}^*}(l-1) )\end{pmatrix}.$$

Note that 1 is true by our inductive hypothesis since $I_{l-1} = I_{l}^{+}$. We will now prove 2. We note that $Y_{j-1}^{i,j}(l) = \frac{\lm_{l} - \lm_j}{\lm_{l-1}-\lm_j}Y_{j-1}^{i,j}(l-1)$ and if $Y_u^{i,j}(l) = \frac{\lm_{l}-\lm_j}{\lm_{l-1}-\lm_j}Y_{u}^{i,j}(l-1)$ then $$Y_{u-1}^{i,j}(l) = \begin{pmatrix} \frac{\lm_{l}-\lm_j}{\lm_{l-1}-\lm_j}Y_{u}^{i,j}(l-1) & -\frac{\lm_{l}-\lm_j}{\lm_{l-1}-\lm_j}Y_u^{i,j}(l-1) \\ 0 & - (\lm_l-\lm_j) ( X_{u}^{i,j}(i) + Y^{i,j}_{u}(u_{ij}^*)) \end{pmatrix} = \frac{\lm_{l}-\lm_j}{\lm_{l-1}-\lm_j}Y_{u-1}^{i,j}(l-1). $$ Therefore we have shown by induction that $Y^{i,j}_u(l) = \frac{\lm_{l} - \lm_j}{\lm_{l-1}-\lm_j}Y_u^{i,j}(l-1)$. 

By our inductive assumption we have that $ Y^{i,j}_{(l-2)_{ij}^*}(l-1) = (n_{S T\cup \{l-1\}})_{S,T \in I_{l-2}}$ and $(n_{S \cup \{l-1\} T})_{S,T \in I_{l-2}} = X^{i,j}_{(l-2)_{ij}^*}(i) + Y^{i,j}_{(l-2)_{ij}^*}(l-1)$ . Let $S= \{s_1,\dots,s_r\}$ and $T = \{t_1,\dots, t_m\}$ be sets in $I_{l-2}$, if there exists $1 \leq k \leq m$ such that  $s_r = t_{k}$ then $$n_{S T\cup \{l-1\}} v^{+}  = (-1)^{m-k+1} (\lm_{l-1}-\lm_j) E_{is_1} \dots E_{s_{r-1}s_r} F_{i t_1} \dots F_{t_{k-1} t_k} v^{+} $$ otherwise $n_{S T\cup \{l-1\}} v^{+}  = 0$. Furthermore, $$n_{S \cup \{l-1\} T} v^{+} = E_{is_1} \dots E_{s_r l-1} E_{l-1 j} F_{it_1} \dots F_{t_m j}v^{+} .$$ We can write 
\begin{align*}
n_{ST\cup\{l\}} v^{+} &= \begin{cases} 0 & s_r \notin T \\ (-1)^{m-k+1} (\lm_{l}-\lm_j) E_{is_1} \dots E_{s_{r-1}s_r} F_{i t_1} \dots F_{t_{k-1} t_k} v^{+} & \exists k, s_r = t_{k} \in T \end{cases} \\
& = \frac{\lm_l-\lm_j}{\lm_{l-1}-\lm_{j} }n_{S T\cup \{l-1\}} v^{+} \\
n_{ST\cup\{l-1,l\}}  v^{+}&= \begin{cases} 0 & s_r \notin T \\ (-1)^{m-k} (\lm_{l-1}-\lm_j) E_{is_1} \dots E_{s_{r-1}s_r} F_{i t_1} \dots F_{t_{k-1} t_k} v^{+} &  \exists k, s_r = t_{k} \in T \end{cases} \\
& = - \frac{\lm_l-\lm_j}{\lm_{l-1}-\lm_j} n_{S T\cup \{l-1\}} v^{+} \\
n_{S\cup \{l-1\} T\cup\{l\}} v^{+}&=  E_{is_1} \dots E_{s_{r} l-1} E_{l-1 j} F_{i t_1} \dots F_{t_m l} E_{l-1 l} F_{l j} v^{+}  \\
& = 0 \\
n_{S\cup \{l-1\} T\cup\{l-1,l\}} v^{+}&= E_{is_1} \dots E_{s_r l-1} E_{l-1j} F_{it_1} \dots F_{t_m l-1}F_{l-1l}F_{lj} v^{+}\\ & = - (\lm_l-\lm_j) E_{is_1} \dots E_{s_r l-1}F_{it_1} \dots F_{t_m l-1} v^{+}  \\ 
& = -(\lm_l-\lm_j) E_{is_1} \dots E_{s_r l-1} E_{l-1 j} F_{it_1} \dots F_{t_m j} v^{+} \\
&=- (\lm_l-\lm_j) (n_{S \cup \{l-1\} T}) v^{+} .
\end{align*} 

Combing these results we get that $$(n_{ST})_{S \in I_l^+ T \in I_l^{-}} = \begin{pmatrix} Y^{i,j}_{(l-2)_{ij}^*}(l) & - Y^{i,j}_{(l-2)_{ij}^*}(l) \\ 0 & -(\lm_l - \lm_j) (X^{i,j}_{(l-2)_{ij}^*}(i) + Y^{i,j}_{(l-2)_{ij}^*}(l-1)) \end{pmatrix} = Y^{i,j}_{(l-1)_{ij}^*}(l).$$

We will now show 3. As before let $S= \{s_1,\dots,s_r\}, T = \{t_1,\dots, t_m\}$ be sets in $I_{l-1}$. We have \begin{align*} n_{S \cup \{l\}T} v^{+} &= E_{i s_1} \dots E_{s_r l} E_{l j} F_{i t_1} \dots F_{t_m j}  v^{+}\\
&= E_{i s_1} \dots E_{s_r l} F_{i t_1} \dots F_{t_m l} v^{+} \\
& = \begin{cases} 0 & s_r \notin T \hspace{0.3 cm}  \text{and} \hspace{0.3 cm} s_r <t_m \\
 E_{i s_1} \dots E_{s_{r-1} s_r} F_{i t_1} \dots F_{t_{m-1} t_m} F_{t_m s_r} v^{+} &  s_r \notin T \hspace{0.3 cm}  \text{and} \hspace{0.3 cm} t_m < s_r \\
(-1)^{m-k} (\lm_{t_m} - \lm_l) E_{is_1} \dots E_{s_{r-1}s_r} F_{i t_1} \dots F_{t_{k-1} t_k} v^{+} & \exists k, s_r = t_{k} \in T \end{cases} \\
& = n_{ST} v^{+} + n_{S T \cup \{l\} }v^{+} .
\end{align*}

Since we have already shown $(n_{S T \cup \{l\} })_{S,T \in I_{l-1}} = (n_{ST}) _{S \in I_l^+ T \in I_l^{-}} = Y^{i,j}_{(l-1)_{ij}^*}(l)$ and $(n_{ST})_{S,T \in I_{l-1}} = X^{i,j}_{(l-1)_{ij}^*}(i)$. From this we can conclude 3. 

Finally to show 4 we again let $S= \{s_1,\dots,s_r\}, T = \{t_1,\dots, t_m\}$ be sets in $I_{l-1}$. We have \begin{align*} n_{S \cup \{l\} T \cup \{l\}} & = E_{i s_1} \dots E_{s_r l} E_{l j} F_{i t_1} \dots F_{t_m l } F_{l j}  v^{+}\\ & = (\lm_l-\lm_j) E_{is_1} \dots E_{s_r l} F_{i t_1} \dots F_{t_m l} v^{+} = (\lm_l-\lm_j)  n_{S \cup \{l\}T} v^{+}. 
\end{align*} 
This completes the proof that $N^{i,j}_{l} = X_{l_{i,j}^*}^{i,j}(i)$ and hence by induction $N^{i,j}= X^{ij}_{i}(i)$.

We now consider the case where $\mu_i = \lm -  \alpha(i,s+1) - \alpha(s, s+2)$. Recall that we denote $X^{i}_u(k) := X^{i,s+1}_u(k)$ and $Y^i_u(k) := X^{i,s+1}$. We can again describe the elements of $I_{\mu}(\lm)$ in terms of sets $S$, now in the interval $(i, s)$. Note that $F_{s+1s+2} v^{+} =0$ since $\lm_{s+1} =0$. Therefore we only wish to consider three distinct possibilities for elements of $I_{\mu}(\lm)$ , namely, \begin{align*}  F_{S x_1} &:= F_{is_1} \dots F_{s_r s+1} F_{s s+2} \\ F_{S x_2} &:= F_{is_1} \dots F_{s s+1} F_{s_r s+2} \\  F_{S x_3} &:= F_{is_1} \dots F_{s_r s} F_{s s+1}F_{s s+2}. \end{align*}

Let $X_{\mu}$ be the matrix $(m_{Sx_a Tx_b})_{S,T \in (i,s), a,b \in \{1,2,3\}}$ where the subsets of $(i,s)$ are ordered as in the previous part of the proof. We wish to show $X_{\mu} = M^i$.  For $S= \{s_1,\dots,s_r\}, T = \{t_1,\dots, t_m\}$ in $(i,s)$  let $n_{ST}$ be the values defined above (for the case where $\mu = \alpha(i,s+1)$). We have, \begin{align*}
m_{Sx_1Tx_1} v^{+} & = E_{is_1} \dots E_{s_{r}s+1} E_{s s+2} F_{i t_1} \dots F_{t_m s+1} F_{ss+2} v^{+} \\
& = \lm_s (n_{ST}) v^{+}\\
& = E_{is_1} \dots E_{s s+1} E_{s_r s+2} F_{i t_1} \dots F_{s s+1} F_{t_m s+2} v^{+} \\
&= m_{Sx_2Tx_2} v^{+} \\ 
m_{Sx_1Tx_2} v^{+} & = E_{is_1} \dots E_{s_{r}s+1} E_{s s+2} F_{i t_1} \dots F_{s s+1} F_{t_m s+2} v^{+} \\
& = n_{ST\cup \{s\}} v^{+}  \\
& = E_{is_1} \dots E_{s s+1} E_{s_r s+2} F_{i t_1} \dots F_{t_m s+1} F_{s s+2} v^{+} \\
& = m_{Sx_2Tx_1} v^{+} \\
m_{Sx_1 Tx_3} v^{+} & =  E_{is_1} \dots E_{s_{r}s+1} E_{s s+2} F_{it_1} \dots F_{t_m s} F_{s s+1}F_{s s+2}v^{+}  \\ 
& = (\lm_s -1) n_{S T \cup \{s\}} v^{+} \\
&=  E_{is_1} \dots E_{s s+1} E_{s_r s+2} F_{it_1} \dots F_{t_m s} F_{s s+1}F_{s s+2} v^{+} \\
&= m_{Sx_2Tx_3} v^{+} \\
m_{Sx_3 Tx_1} v^{+} & = E_{is_1} \dots E_{s_r s} E_{s s+1}E_{s s+2} F_{i t_1} \dots F_{t_m s+1} F_{s s+2} v^{+} \\
&= E_{is_1} \dots E_{s_r s} E_{s s+2}F_{i t_1} \dots F_{t_m s} F_{s s+2} v^{+} \\
&= n_{S \cup \{s\} T \cup \{s\}} v^{+} \\
&= E_{is_1} \dots E_{s_r s} E_{s s+1}F_{i t_1} \dots F_{t_m s} F_{s s+1} v^{+} \\
&=  E_{is_1} \dots E_{s_r s} E_{s s+1}E_{s s+2} F_{i t_1} \dots F_{s s+1} F_{t_m s+2} v^{+} \\
&= m_{Sx_3 Tx_2} v^{+} \\
m_{Sx_3 Tx_3} v^{+} & = E_{is_1} \dots E_{s_r s} E_{s s+1}E_{s s+2} F_{i t_1} \dots F_{t_m s} F_{s s+1} F_{s s+2} v^{+} \\
& = (\lm_s -1) n_{S \cup \{s\} T \cup \{s\}} v^{+} . 
\end{align*} From this we can conclude $X_{\mu} = M^i$.

Finally, let us consider the case where $\mu = \lm - \alpha(i,s+1) - \alpha(s,s+1)$. The elements of $I_{\mu}(\lm)$ are given by sets $S$ in $(i,s)$ and an elements $a \in \{1, 2\}$ where, for $S = \{s_1 \dots s_r\}$, we have \begin{align*} F_{Sx_1} &:= F_{is_1} \dots F_{s_r s+1} F_{s s+1} \\ F_{Sx_2} &:= F_{is_1} \dots F_{s_r s} F^{(2)}_{s s+1} . \end{align*}

Let $X_{\mu}$ be the matrix $(w_{Sx_aTx_b})_{S,T \in (i,s), a,b \in \{1,2\}}$ with the subsets ordered as before. For $S= \{s_1,\dots,s_r\}, T = \{t_1,\dots, t_m\}$ in $(i,s)$ let $n_{ST}$ be the value defined above. We have, 
\begin{align*} w_{Sx_1 Tx_1} v^{+} &= E_{is_1} \dots E_{s_r s+1} E_{s s+1} F_{i t_1} \dots F_{t_m s+1} F_{s s+1} v^{+} \\
&= (\lm_s n_{ST}) + n_{ST\cup\{s\}} v^{+} \\
w_{Sx_1 Tx_2} v^{+} &= E_{is_1} \dots E_{s_r s+1} E_{s s+1} F_{i t_1} \dots F_{t_m s} F^{(2)}_{s s+1} v^{+} \\
&= (\lm_{s} - 1) n_{ST\cup\{s\}} v^{+}  \\
w_{Sx_2 Tx_1} v^{+} &= E_{is_1} \dots E_{s_r s} E^{(2)}_{s s+1} F_{i t_1} \dots F_{t_m s+1} F_{s s+1} v^{+} \\
&= \lm_s (n_{ST} + n_{ST\cup\{s\}}) v^{+} \\
w_{Sx_2 Tx_2} v^{+} &= E_{is_1} \dots  E_{s_r s} E^{(2)}_{s s+1} F_{i t_1} \dots F_{t_m s} F^{(2)}_{s s+1} v^{+} \\
& = \frac{1}{2} \lm_s(\lm_s-1) ( n_{ST} + n_{ST\cup\{s\}} ) v^{+} . \\
\end{align*} From this we can conclude $X_{\mu} = W^i$ and the proof is complete. 

 \end{proof}

We can now define the composition multiplicities in terms of the values $m_s^i$, $w_s^i$ and $x_{i,j}$. Let $t_{s}^{s-1} := m_{s}^{s-1} - w_{s}^{s-1}$ and for $1 \leq i <s-1$ define $t_s^i := (m_s^i - w_s^i) - \sum_{r=i+1}^{s-1} t_s^r x_{i,r}$. 

\begin{thm} \label{multnm2}
Let $\lm$ and $\mu_i$ be as defined above with $\lm_s$ the last nonzero entry of $\lm$ and suppose $p \neq 2$. Then, $$ [L_n(\lm)^{(1,1) \times} : L_{n-2}(\omi) ] = t_s^i.$$
\end{thm} 

The proof of Theorem \ref{multnm2} will follow from the results above and the following proposition.

\begin{prop}\label{weightsabovemu}
Let $\nu$ be the $\GL(n)$-weight of a nonzero weight vector in $L_n(\lm)^{(1,1) \times}$ such that $\omi \leq \on$ and $L_{n-2}(\on)$ is a composition factor of $L_n(\lm)^{(1,1)\times}$. Then $\nu = \mu_j$ for some $j \geq i$.
\end{prop}

\begin{proof}
Since $\nu$ corresponds to a nonzero weight vector in $L_n(\lm)^{(1,1) \times}$ we can write $\nu = \lm - \sum_{r=1}^{n-1}n_r \alpha_r $ where $n_{n-2} = 2$ and $n_{n-1} =1$. Since $\lm_{n} =0$ and therefore the weights of $L_n(\lm)$ have nonnegative entries we note that whenever $\lm_r =0$ we have $n_{r-1} \geq n_r$. Recall that $\lm_r =0$ for $s+1 \leq r \leq n$ and therefore $n_r \geq 2$ for $s \leq r \leq n-3$.

We can write, $$ \on - \omi = -\sum_{r=1}^{i-1} n_r \alpha_r + \sum_{r=i}^{s-1} (1-n_r) \alpha_r + \sum_{r=s}^{n-3} (2-n_r) \alpha_r. $$ Since we assume $\omi \leq \on$ it follows that $n_r = 2$ for $s \leq r \leq n-3$, $n_r \leq 1$ for $i \leq r \leq s-1$ and finally $n_r =0$ for $1 \leq r \leq i-1$. Our result follows if we can show there exists some $i \leq j \leq s-1$ such that $n_r = 0$ for $i \leq r \leq j-1$ and $n_r = 1$ for $j \leq r \leq s-1$. Let $n_j$ be the first nonzero entry in $ (n_i,\dots,n_{s-1})$. We will show that $n_{s-1} =1$ and that there can be no zeros between $n_j$ and $n_{s-1}$. 

Let us first show that $n_{s-1} = 1$. If $\lm_s =1$ then this follows again from the fact that there are no nonzero weight spaces in $L_n(\lm)$ whose weights have negative entries. Now let us suppose $\lm_s >1$ and assume for contradiction that $n_{s-1} = 0$. Let $\omega = (\nu_1, \dots, \nu_{n-2}, 2,0)$ and define $\nu' = (\nu_1,\dots, \nu_{s},1,1,0,\dots,0)$ and $\omega' = (\nu_1,\dots, \nu_{s},2,0,0,\dots,0)$. Clearly $\omega$ and $\nu$ are in the same Weyl orbit as $\omega'$ and $\nu'$ respectively. Note that $\nu' = \lm - \sum_{r=1}^s n_r \alpha_r - \alpha_{s+1}$ and $\omega' =  \lm - \sum_{r=1}^s n_r \alpha_r$.

Observe that any element of $I_{\nu'}(\lm)$ can be written as $F_{\underline{x}}F_{ss+1}F_{ss+2}$ where any $F_{ij}^{(k)} \in F_{\underline{x}}$ is such that $j <s$. This follows from the fact that $\nu_s = \lm -2$ since $n_{s-1} =0$ and the fact that $F_{s+1}F_{s+2} v^{+} =0$. Let $I'_{\nu'}(\lm)$ be the set of all $F_{\underline{x}}$ such that $F_{\underline{x}}F_{ss+1}F_{ss+2} \in I_{\nu'}(\lm).$ Any element of $I_{\omega'}(\lm)$ can be written as $F_{\underline{x'}}F^{(2)}_{ss+1}$ where any $F_{ij}^{(k)} \in F_{\underline{x'}}$ is such that $j <s$ again since $\nu_s = \lm -2$.  Let $I'_{\omega'}(\lm)$ be the set of all $F_{\underline{x'}}$ such that $F_{\underline{x'}}F^{(2)}_{ss+1} \in I_{\omega'}(\lm)$. Since $\omega_i' = \nu_i = \nu_i'$ for $1 \leq i \leq s$ it follows that $I'_{\omega'}(\lm) = I'_{\nu'}(\lm)$. 

We have seen in the proof of Theorem \ref{rankdims} that $X_{\nu'} = \lm_s^2 - \lm_s (n_{\underline{x}\underline{y}})_{F_{\underline{x}},F_{\underline{y}} \in I'_{\nu'}(\lm)}$ and $X_{\omega'} = \frac{1}{2} (\lm_s^2 - \lm_{s})(n_{\underline{x}\underline{y}})_{F_{\underline{x}},F_{\underline{y}} \in I'_{\nu'}(\lm)}$. Noting that $p$ divides $(\lm_s^2-\lm_s)$ if and only if $p$ divides $ \frac{1}{2} (\lm_s^2 - \lm_{s})$ we deduce that the mod $p$ rank of $X_{\nu}$ and $X_{\omega'}$ are the same. This implies that $\dim(L_n(\lm)_{\nu}) = \dim(L_n(\lm)_{\nu'}) \dim(L_n(\lm)_{\omega'}) = \dim(L_n(\lm)_{\omega}) $ and hence by Theorem \ref{weightsspaces} we have that $\dim((L_n(\lm)^{(1,1) \times})_{\nu}) =0$. This is a contradiction and therefore we may assume $n_{s-1} =1$.

Now suppose that there exists some $j< m <s-1$ such that $n_m =0$ and $n_r =1$ for all $m <r \leq s-1$. We will show $L_{n-2}(\on)$ is not a composition factor of $L_n(\lm)^{(1,1) \times}$ (we believe this to be a well know result and the idea for this proof is taken from Lemma 7.4 in \cite{KLESH1997} but adapted for our circumstance). Suppose for a contradiction that it is. Then there exists some $0 \neq v \in L_n(\lm)_{\nu}$ which is $\GL(n-2)$-primitive. For any sequence $S = [s_1, \dots, s_f]$ let $F_S = F_{s_1 s_2} F_{s_3 s_4} \dots F_{s_{f-1} s_{f}}$. Suppose $S, S'$ are sequences of elements from $ \{j, \dots, m\}$ and $T , T'$ are sequences of possibly repeated elements from $\{m+1, \dots, n\}$. We can write $v = F_S F_T v^{+}$ where $F_S F_T v^{+} = F_T F_S v^{+}$ and $v^{+} = E_{S'} E_{T'} v$. Therefore we have $E_{S'} E_{T'} F_SF_T v^{+} = v^{+}$ and in particular $E_{T'} F_{T} E_{S'}F_S v^{+} \neq 0$. This implies $E_{S'}F_S v_{+} = cv^{+}$ where $c \neq 0$. However $v$ is $\GL(n-2)$-primitive and therefore $E_{S'} v = 0$. Therefore \begin{align*} 0 = F_S E_{S'} v = F_S E_{S'} F_{S} F_{T} v^{+} = F_S F_T E_{S'} F_S v^{+} = c (F_S F_T )v^{+} = cv \end{align*} and hence $c=0$ which is a contradiction. Therefore no such $m$ can exist and hence there can be no zeros between $n_j$ and $n_{s-1}$. 
\end{proof} 

\begin{proof}[Proof of Theorem \ref{multnm2}] 
We have seen above that $m_s^i -  w_s^i$ is the dimension of $(L_n(\lm)^{(1,1) \times})_{\omi}$ for all $1 \leq i <s$.  Furthermore, for $i < j <s$ we have $\omi = \omj - \alpha(i,j)$ and hence $\dim((L_{n-2}(\omj))_{\omi}) = x_{i,j}$. Therefore the result follows from Proposition \ref{weightsabovemu} by counting weight spaces. 
\end{proof}

Combining Theorems \ref{multnm2} and \ref{key} we can now deduce Theorem \ref{main}.

\begin{proof}[Proof of Theorem \ref{main}]
The conditions $1$ and $2$ of Theorem \ref{key} are clearly met by the $\lm$ and $\mu$ considered. Let $\alpha \in \{0,\dots,p-1\}$ be such that $\alpha \equiv p-n \mod p$. Note that $cont_{\alpha}(\lm - \epsilon_r - \epsilon_n) < cont_{\alpha}(\lm)$ for $(r,n)$ a $\lm$-removable pair. However since $\lm_s - s$ and $\lm_i - i$ are not congruent to $\alpha$ we see that $cont_{\alpha}(\lm - \epsilon_i- \epsilon_s) = cont_{\alpha}(\lm)$. Therefore condition $3$ of Theorem \ref{key} is also met in this case. The result is then a direct consequence of Theorems \ref{multnm2} and \ref{key}. 
\end{proof}

Let us finish with an example and concrete result about the multiplicities of certain composition factors of $L_n(\lm) \otimes \bigwedge\nolimits^2 V_n^*$. 

\begin{exmp}
Let $\lm$ be as in Theorem \ref{main} with $i = s-1$. Then, $$[L_n(\lm) \otimes \bigwedge\nolimits^2 V^*_n : L_n(\lm - \epsilon_i - \epsilon_s)] = \begin{cases} 0 & p | \lm_s \hspace{0.5cm} \text{or} \hspace{0.5cm} p | \lm_{s-1} + 1 \\ 1 & \text{otherwise}. \end{cases}$$
\end{exmp}

\begin{proof}
Note that $t_s^{s-1} = m_s^{s-1} - w_s^{s-1}$ therefore it remains only to calculate the values. We have, $$M^{s-1} = \begin{pmatrix} \lm_s \lm_{s-1} & -\lm_s & -(\lm_s -1) \lm_s \\ 
-\lm_s & \lm_s \lm_{s-1} & -(\lm_s -1) \lm_s \\
\lm_s(\lm_{s-1} - \lm_{s}) & \lm_s(\lm_{s-1} - \lm_{s}) & \lm_s( \lm_s-1) (\lm_{s-1} - \lm_{s}) \end{pmatrix}$$ and $$W^{s-1} = \begin{pmatrix} \lm_s (\lm_{s-1} -1) & -(\lm_s -1) \lm_s \\ 
\lm_s(\lm_{s-1} - \lm_{s}) & \frac{1}{2} \lm_s( \lm_s-1) (\lm_{s-1} - \lm_{s}) \end{pmatrix}.$$
Therefore $$m_s^{s-1}  = \begin{cases} 0 & p|\lm_s \\  1 & p| \lm_{s-1} +1 \hspace{0.3cm} \text{or} \hspace{0.3cm} (p| \lm_{s-1} -\lm_s  \hspace{0.3cm} \text{and} \hspace{0.3cm} p| \lm_s -1 )\\ 2 & p| \lm_{s-1} - \lm_s  \hspace{0.3cm} \text{or} \hspace{0.3cm} p | \lm_s -1 \\ 0 & \text{otherwise} \hspace{0.3cm} \end{cases} $$ and $$w_s^{s-1} = \begin{cases} 0 & p|\lm_s \hspace{0.3cm} \text{or} \hspace{0.3cm} (p | \lm_s -1 \hspace{0.3cm} \text{and} \hspace{0.3cm} p| \lm_{s-1} - \lm_s) \\ 1 & p| \lm_{s-1} - \lm_s \hspace{0.3cm} \text{or} \hspace{0.3cm} p| \lm_s -1 \hspace{0.3cm} \text{or} \hspace{0.3cm} p | \lm_{s-1} +1 \\ 2 & \text{otherwise} \hspace{0.3cm} \end{cases} $$
Subtracting these values and applying Theorem \ref{main} we get the stated result. 

\end{proof}

\bibliography{Bibliography}

\bibliographystyle{amsplain}

\end{document}